\documentclass[10pt,reqno]{amsart}
\usepackage[nodisplayskipstretch]{setspace}
\usepackage[brazilian,english]{babel} 
\usepackage[T1]{fontenc}
\usepackage{amsfonts,amsthm,amssymb,amsmath,mathrsfs}  
\usepackage{graphicx,color}       
\usepackage{latexsym}       
\usepackage{overpic}        
\usepackage[colorlinks=false, linktocpage=true]{hyperref}
\usepackage{graphicx}
\usepackage{multicol}
\usepackage{color}
\usepackage{multirow} 
\usepackage{mathtools}
\usepackage{xcolor}
\mathtoolsset{showonlyrefs}

\newtheorem{theorem}{Theorem}[section]
\newtheorem{corollary}[theorem]{Corollary}

\newtheorem{lemma}[theorem]{Lemma}
\newtheorem{proposition}[theorem]{Proposition}

\theoremstyle{definition}
\newtheorem{definition}[theorem]{Definition}
\newtheorem{remark}[theorem]{Remark}

\newcommand{\bu}{\bar{u}}
\newcommand{\E}{\mathcal{E}}

\newcommand{\R}{\mathbb{R}}
\newcommand{\N}{\mathbb{N}}
\newcommand{\C}{\mathbb{C}}

\newcommand{\Pc}{\mathcal{P}}

\newcommand{\Dz}{\dot{\mathbf{H}}^{2}(\R^{d})}
\newcommand{\dz}{\dot{\mathbf{H}}^{2}}
\newcommand{\psib}{\mbox{\boldmath$\psi$}}
\newcommand{\phib}{\mbox{\boldmath$\phi$}}
\newcommand{\phis}{\phi\left(\frac{|x|}{R}\right)}

\newcommand{\ub}{\mathbf{u}}
\newcommand{\uck}{u_{ck}}
\newcommand{\partialj}{\partial_{x_j}}
\newcommand{\ubc}{\mathbf{u}_{c,0}}
\newcommand{\vb}{\mathbf{v}}
\newcommand{\zb}{\mathbf{z}}
\newcommand{\gb}{\mathbf{g}}
\newcommand{\Hb}{\mathbf{H}}
\newcommand{\wb}{\mathbf{w}}
\newcommand{\Lb}{\mathbf{L}}
\newcommand{\Wb}{\mathbf{W}}

\newcommand{\Mb}{\mathbf{M}}
\newcommand{\Bb}{\mathbf{B}}
\newcommand{\Nb}{\mathbf{N}}
\newcommand{\Zb}{\mathbf{Z}}
\newcommand{\Cb}{\mathbf{C}}

\newcommand{\kl}{k=1,\ldots,l}
\newcommand{\sumk}{\sum_{k=1}^l}
\newcommand{\sumj}{\sum_{j=1}^l}
\newcommand{\summ}{\sum_{m=1}^l}
\newcommand{\sumjd}{\sum_{j=1}^d}

\usepackage{a4wide}
\usepackage{mathabx}

\newcommand{\RE}{\mathrm{Re}}
\newcommand{\IM}{\mathrm{Im}}

\newcommand{\bbc}{\!\!\big\bracevert\!\!}

\newcommand{\Mmb}{\mathcal{M}_+^b(\R^d)}
\newcommand{\Mmu}{\mathcal{M}_+^1(\R^d)}
\newcommand{\cfe}{\overset{\ast}{\rightharpoonup}}
\newcommand{\chir}{\chi_R^r}
\newcommand{\xR}{\left(\frac{|x|}{R}\right)}
\newcommand{\gk}{\gamma_k}
\newcommand{\ak}{\alpha_k}
\newcommand{\V}{\mathcal{V}}
\newcommand{\GR}{\Gamma_{\phi_R}}

\numberwithin{equation}{section}

\begin{document}
	\title[Biharmonic NLS system]{ENERGY CRITICAL FOURTH-ORDER SCHRÖDINGER EQUATION SYSTEM WITH POWER-TYPE NONLINEARITIES IN THE RADIAL CASE}
	
	
	\author[M. Hespanha]{Maicon Hespanha}
    \author[R. Scarpelli]{Renzo Scarpelli}\address{ICEx, Universidade Federal de Minas Gerais, Av. Antônio Carlos, 6627, Caixa Postal 702, 30123-970 Belo Horizonte, MG, Brazil}
	\email{mshespanha@gmail.com}
    \email{renzoscb123@gmail.com}
	
	
	\begin{abstract}
  In this paper, we study a system of focusing fourth-order Schrödinger equations in the energy-critical setting with radial initial data and general power-type nonlinearities. The main idea is to generalize the analysis of such systems: we first establish several hypotheses on the nonlinearities and prove their implications. These implications are then used to establish a local well-posedness result in $H^2(\mathbb{R}^d)$ and and to prove the existence of ground state solutions. Using a virial argument, we demonstrate a blow-up result for initial data with negative energy or with kinetic energy exceeding that of the ground state. Finally, employing the concentration-compactness/rigidity method, we prove a scattering result for solutions whose energy and kinetic energy are below those of the ground state.
	
	\end{abstract}
	
	\maketitle
	
	\section{Introduction}

The fourth-order Schrödinger equation 
\begin{equation}\label{karpmanequa}
	i\partial_t u+\Delta^2 u+\epsilon \Delta u+\lambda |u|^{p-1}u=0
\end{equation}
was introduced by \cite{Karpman} and \cite{Karpmn2} to consider the role of small fourth-order dispersion terms in the propagation of laser beams in with Kerr nonlinearity. The author established  necessary and sufficient conditions of the stability with respect to small perturbations and obtained necessary conditions of the stabilization of self-focusing and blow-up by high-order dispersion. In \cite{AKS}, the authors proved sharp dispersive estimates for the linear group associated to $i\partial_t +\Delta^2+\epsilon \Delta$, with $\epsilon\in\R$ essentially given by $-1,\,0,$ or $1$. In \cite{FIP}, the authors described various properties of the equation in subcritical regime, with part of the analysis relying on numerical developments.
When $d\geq 5$, and $f(u)=|u|^{\frac{8}{d-4}}u$, the criticality in the energy space appears with the power exponent $2d/(d-4)$ for the embedding of $H^2$ into Lebesgue's spaces. This case was studied in \cite{Pausader}, where was proved local well-posedness and stability in a general setting and global well-posedness and scattering in the defocusing ($\lambda >0$) case for radially symmetrical initial data. In \cite{Pausader2}, the author investigated the focusing ($\lambda <0$) energy-critical fourth-order Schrödinger equation in the radial setting and proved existence and scattering for solutions of energy and $\dot{H}^2$-norm below that of the ground state. A similar result was obtained independently by \cite{MXZ}.  The global well posedness and scattering in $H^s$ for small data it was proved in \cite{GW}. In \cite{HHW1} and \cite{HHW2}, the authors discussed the Cauchy problem in a high-regularity setting.

In this paper we will consider the energy-critical case for the following Cauchy problem to the fourth-order coupled nonlinear Schrödinger system in the homogeneous $(\epsilon=0)$ and focusing $(\lambda<0)$ regime.
\begin{equation}\label{SISTB}
	\left\{\begin{array}{ll}
		i\alpha_k\partial_t u_k +\gamma_k \Delta^2 u_k - f_k(\mathbf{u})=0,\\
		(u_1(0,x),...,u_l(0,x))=(u_{10},...,u_{l0})\in \Dz,\quad k=1,...,l,
	\end{array}\right.
\end{equation}
where $\ub=(u_1,...,u_l)$ is a complex-valued function in the space-time $\R\times\R^d$, with $5\leq d\leq 16$,  $\Delta^2$ stands for the biharmonic operator and $\alpha_k, \gamma_k>0$. Also, the nonlinearities $f_k:\C^l\rightarrow \C$ satisfy a power-type growth and are some suitable assumptions that will be displayed later. One may notice that we have a constraint in the dimension, the lower bound is a consequence of the energy-critical regime and the power of the nonlinearities. The upper bound, is a consequence of a necessary condition to apply Strichartz-type estimates (see Proposition \ref{strichartz} and Remark \ref{dimension}).

Motivated by the work in NLS systems made in \cite{NoPa}-\cite{NoPa5}, the idea of this paper is provide sufficient conditions on the nonlinearities $f_k$ to study the dynamics of a system of fourth-order nonlinear Shcrödinger equations with general power-type nonlinearities. Our intention here is study the energy critical regime, that is, the exponent of the power-type growth will be $\frac{d+4}{d-4}$. As mentioned before, to reach our goal will be necessary assume suitable conditions on the nonlinear terms $f_k$. Therefore, in what follows we will assume  

	\newtheorem{thmx}{}
\renewcommand\thethmx{(H1)}
\begin{thmx}\label{H1}
	We have
	\begin{align*}
		f_{k}(\mathbf{0})=0, \qquad  k=1,\ldots,l.
	\end{align*}
\end{thmx}

\renewcommand\thethmx{(H2)}
\begin{thmx}\label{H2}
	For any $\mathbf{z},\mathbf{z}'\in \C^{l}$ we have
	\begin{equation*}
		\begin{split}
			\left|\frac{\partial }{\partial z_{m}}[f_{k}(\mathbf{z})-f_{k}(\mathbf{z}')]\right|+ \left|\frac{\partial }{\partial \overline{z}_{m}}[f_{k}(\mathbf{z})-f_{k}(\mathbf{z}')]\right|&\leq C \sum_{j=1}^{l}|z_{j}-z_{j}'|^{\frac{8}{d-4}},\qquad k,m=1,\ldots,l,
		\end{split}
	\end{equation*}
\end{thmx}

\renewcommand\thethmx{(H3)}
\begin{thmx}\label{H3}
	There exists a function $F:\C^{l}\to \C$,  such that
	\begin{equation*}
		f_{k}(\mathbf{z})=\frac{\partial F}{\partial \overline{z}_{k}}(\mathbf{z})+\overline{\frac{\partial F }{\partial z_{k}}}(\mathbf{z}),\qquad k=1,\ldots,l.
	\end{equation*}
\end{thmx}

\renewcommand\thethmx{(H4)}
\begin{thmx}\label{H4}
	For any $\theta\in\R$ and $\zb\in\C^l$,
	$$
	\RE F(e^{\frac{\alpha_1}{\gamma_1}}z_1,...,e^{\frac{\alpha_l}{\gamma_1}}z_l)=\RE F(\zb)
	$$
	
\end{thmx}

\renewcommand\thethmx{(H5)}
\begin{thmx}\label{H5}
	Function $F$ is homogeneous of degree $\frac{2d}{d-4}$, that is, for any $\mathbf{z}\in \mathbb{C}^{l}$ and $\lambda >0$,
	\begin{equation*}
		F(\lambda \mathbf{z})=\lambda^{\frac{2d}{d-4}}F(\mathbf{z}).
	\end{equation*}
\end{thmx}

\renewcommand\thethmx{(H6)}
\begin{thmx}\label{H6}
	There holds
	\begin{equation*}
		\left|\mathrm{Re}\int_{\R^{d}} F(\ub)\;dx\right|\leq \int_{\R^{d}} F(\!\!\big\bracevert\!\! \mathbf{u}\!\!\big\bracevert\!\!)\;dx.
	\end{equation*}
\end{thmx}

\renewcommand\thethmx{(H7)}
\begin{thmx}\label{H7}
	Function $F$ is real valued on $\R^l$, that is, if $(y_{1},\ldots,y_{l})\in \R^{l}$ then
	\begin{equation*}
		F(y_{1},\ldots,y_{l})\in \R.
	\end{equation*}
	Moreover, functions	$f_k$ are non-negative on the positive cone in $\mathbb{R}^l$, that is, for $y_i\geq0$, $i=1,\ldots,l$,
	\begin{equation*}
		f_{k}(y_{1},\ldots,y_{l})\geq0.
	\end{equation*}	
\end{thmx}

Summarizing, we use hypothesis \ref{H1} and \ref{H2} to prove local well-posedness for \eqref{SISTB}. The Mass and Energy of \eqref{SISTB} are, respectively, given by
\begin{equation}\label{mass}
	M(\ub)=	\sumk\frac{\alpha_k^2}{\gamma_{k}}\Vert u_k\Vert_{L^2}
\end{equation}
and 
\begin{equation}\label{energy}
	E(\ub)=\frac{1}{2}\sumk \gamma_k \Vert \Delta u_k\Vert_{L^2}^2  -\RE \int F(\ub)dx:=\frac{1}{2}K(\ub)-P(\ub),
\end{equation}
where $K$ and $P$ are respectively called by kinetic and potential energy. To prove conservation of the mass and the energy its enough to assume \ref{H3}-\ref{H5} (see \cite[Lemma 2.1]{FH} for details). 

We start by defining the concept of solution to \eqref{SISTB}. 
\begin{definition}\label{solution}
	By a solution of \eqref{SISTB}, we will understand $\ub\in \mathbf{C}_t\Hb^2(I\times\R^d)$, where $I\subset \R$ is a non-empty time interval containing $0$ and satisfies the Duhamel formula
	\begin{equation}\label{defsol}
		u_k(t)=U_k(t)u_{k0}-i\int_0^tU_k(t-s)f_k(\ub(s))ds,\quad \kl,\,t\in I
	\end{equation}
	where $U_k(t)=e^{it\frac{\gamma_k}{\alpha_k}\Delta^2}$, $\kl$, is the biharmonic Schrödinger propagator. The time interval $I$ is said to be the lifespan of $\ub$. We say that $\ub$ is a maximal-lifespan solution if the solution cannot be extended to an interval $J\supset I$ strictly larger than $I$. We say that the solution is global if $I=\R$.
\end{definition}
It is easy to see that for $\nu>0$, the equation \eqref{SISTB} is preserved by the scaling
\begin{equation}\label{scaling}
f_\nu(t,x)=\nu^{\frac{d-4}{2}}f(\nu^4 t, \nu x), \quad \nu>0.
\end{equation}
 This scaling preserves the homogeneous $\dot{H}^{2}$-norm of the original solution $\ub(t)$. Moreover, the equation \eqref{SISTB} also satisfy time reversal and time translation symmetries. 
 
 Another important concepts that will be discussed throughout the paper is the following notions of Scattering size and blow-up. 
 \begin{definition}[Scattering size]\label{scadef}
 	Let $\mathbf{u}$ be a solution of \eqref{SISTB}. The scattering size of  $\mathbf{u}$ on a time interval $I$ is defined as
 	$$
 	S_I(\mathbf{u}):=\sum_{k=1}^l\int_I\int_{\mathbb{R}^d}|u_k(t,x)|^{\frac{2(d+4)}{d-4}}dxdt
 	$$
 \end{definition}
 
 \begin{definition}[Blow-up] We say that a solution $\mathbf{u}$ of \eqref{SISTB} blows-up forward in time if there exists $t_1\in I$ such that 
 	$$
 	S_{[t_1,\sup I)}(\mathbf{u})=\infty,
 	$$
 	and $\mathbf{u}$ blows-up backward in time, if there exists $t_2\in I$ such that
 	$$
 	S_{(\inf I,t_2]}(\mathbf{u})=\infty.
 	$$
 	We say that $\mathbf{u}$ blows-up in time if it blows-up both forward and backward in time.
 \end{definition}

 Our main interest is to give conditions to deal with the dichotomy global existence versus blow-up in finite time when the initial data is radial. To this, we need to set the threshold for such dichotomy. Is well know, from NLS equation that ground state solutions plays a important role in this kind of studies. So, the next step is to guarantee the existence of such type of solutions.  To give a precise description of this solution we start with standing waves solutions, that are a special kind with the form
 \begin{equation}\label{standingwaves}
 	u_k(t,x)=e^{i\frac{\sigma_k}{2}\omega t}\psi_k(x),\quad\kl,
 \end{equation}
 where $\omega\in\R$ and $\psi_k$ are real-valued functions decaying to zero at infinity. By Lemma \ref{GC}, one may see that the $\psi_k$ satisfy the following semilinear elliptic system
 \begin{equation}\label{eliptic}
 	-\gamma_k\Delta^2\psi_k=f_k(\psib),\quad \kl.
 \end{equation}
 
We will see in Section 4 that solutions of \eqref{eliptic} can be seen as critical points of the energy functional $E$. More precisely

\begin{definition}\label{weaksolution}
	We say that $\psib\in\Dz$ is a weak solution of \eqref{eliptic} if is critical point of $E$, that is, if for any $\gb\in\Dz$,
	\begin{equation}
		\gamma_k\int \Delta \psi_k\Delta g_k=\int f_k(\psib)g_kdx,\quad \kl.
	\end{equation}
	Among all the solutions of \eqref{eliptic}, the ground state solutions is the one that minimizes $E$. Precisely, we say that $\psib\in\Dz$ is a \textit{ground state} solution if
	$$
	E(\psib)=\inf\{E(\phib);\, \phib\in\mathcal{C}\},
	$$ 
	where $\mathcal{C}$ is the set of non trivial solutions of \eqref{eliptic}. We refer the set of ground state solutions by $\mathcal{G}$.
\end{definition} 
 
Our main result in this direction is the following:
\begin{theorem}\label{ESGS} There exists at least one ground state solution $\psib$ for system \eqref{eliptic}, that is, $\mathcal{G}$ is non-empty.
\end{theorem}
As we will see, the ground state solution is direct related with a critical Sobolev-type inequality. From this inequality we will derive a minimization problem, which will be proved by using the Lions concentration-compactness method, introduced in \cite{lions2}. Once this problem is solved, we can prove the existence of ground state solutions. Also, we also obtain a optimal constant for the critical Sobolev inequality. 
	
	Once the existence of ground state is guaranteed, we can set as the threshold for our dichotomy. More specifically, we will see that if either the kinetic energy of the initial data is above of the kinetic energy of the ground state or the energy is negative, the solution must blow-up in finite time. Our result in this topic reads as follows
	
	\begin{theorem}\label{teoblowup} Let $5\leq d\leq 16$. Suppose that $\ub_0\in\Hb^2(\R^d)$ is radial. We assume either $E(\ub_0)<0$ or
		\begin{equation}\label{blowupcondition}
			E(\ub)<E(\psib)\quad\hbox{and}\quad K(\ub_0)>K(\psib),
		\end{equation}
		where $\psib$ is any ground state solution of \eqref{eliptic}, then the solution $\ub\in\mathbf{C}_t\Hb_x^2([0,T)\times\R^d)$ to \eqref{SISTB} blows-up in finite time, i.e., $T<\infty$ and	$	S_{[0,T)}(\ub)=\infty.$
	\end{theorem}
As in NLS case, to prove this kind of result we will apply the convexity method, which consists in deriving a contradiction by working with some virial type equality. For our purpose, we will follow the ideas presented in \cite{BL17} and working with the function 
\[
\V_R(t)=-\sumk\alpha_k\IM\int\nabla\phi_R\cdot\nabla u_k\bu_kdx.
\]
and its derivative $\V'(t)$ in order to reach a contradiction with the length of the existence time interval $I$.

In the last part of the paper we will be concerned with the global existence and scattering of radial solutions in the homogeneous Sobolev space $\Dz$. Here we will prove that if we invert the second inequality in \eqref{blowupcondition}, the solution to \eqref{SISTB} is global and scatters. The results are the following.
\begin{theorem}(Spacetime bounds)\label{spacetimebounds} Let $5\leq d\leq 16$ and $\ub:I\times\R^d\rightarrow\C$ be a radial solution of \eqref{SISTB}. If 	\begin{equation}\label{scatcondition}
	E(\ub_0)<E(\psib)\quad\hbox{and}\quad K(\ub_0)< K(\psib)
	\end{equation}
	Then $S_I(\ub)<\infty.$
\end{theorem}
As a consequence, 
\begin{corollary}
	Let $5\leq d\leq 16$ and let $\ub$ be a maximal lifespan solution to \eqref{SISTB} on the time interval $I$. Assume that \eqref{scatcondition} holds. Then, $I=\R$ and $S_R(\ub)<\infty$. In particular, there exist asymptotic states $\ub^{\pm}\in\Dz$ such that
	\[\Vert \ub(t)-\mathbf{U}(t)\ub^{\pm}\Vert_{\Dz}\rightarrow 0,\quad\hbox{as }r\rightarrow\pm\infty,\]
	where $\mathbf{U}(t)=(U_1(t),....,U_l(t))$.
\end{corollary}

To this end, we will follow the ideas presented in \cite{MXZ}, which is a variant of the concentration-compactness and rigidity method in \cite{killip}, with the fourth order Schrödinger equation settings. See also \cite{KM} and \cite{killip2}. The main difference is that we are working on the setting of coupled nonlinear system with general power-type nonlinearities. 

The whole idea is to use some properties of the nonlinearities $f_k$ to guarantee that the method works for our general system. We start assuming by contradiction that Theorem \ref{spacetimebounds} fails. This allows us to reduce the solution of \eqref{SISTB} to the case of almost periodic solutions. Therefore,  we will see that this kind of solutions has three particular properties, which are called ``the enemies'' in the literature. The next step is to prove that none solutions satisfies this properties. To this end, we prove some adapted versions of the coercivity lemmas in \cite{KM} involving the kinetic and potential energies and establish the energy trapping result for our setting, which are important for precluding the three-enemies scenario. Next, we prove a virial-type inequality, which is necessary to ensure that almost periodic solutions cannot exist, thereby deriving a contradiction. 

The paper is organized as follows. Section 2 is devoted to introduce some notations and preliminary results. In Section 3, we establish consequences of our assumptions \ref{H1}-\ref{H7}. In Section 4, we work on local theory for \eqref{SISTB}. The existence of ground state solutions  is proved in Section 5. In Section 6, we prove virial identities results and the blow-up in finite time. Finally, section 7 is devoted to prove global well-posedness and scattering.

	\section{Notation and preliminary results}\label{secnot}

Throughout the work we will use the standard notation in PDEs. Indeed, $C$ will represent a generic constant which may vary from inequality to inequality. If $a$ and $b$ are positive constants, we denote $a\lesssim b$ whenever  $a\leq Cb$ for some constant $C>0$, similar for the case $a\gtrsim b$. We write $X\pm$ for any quantity of the form $X\pm \epsilon$ for any small $\epsilon>0$. Given a subset $A$, we denote by $\mathbf{A}$ the product $A\times \cdots\times A$ ($l$-times). In particular, if $A$ is a Banach space, then $\mathbf{A}$ also is with the usual norm given by the sum. For a complex number $z\in \mathbb{C}$, Re$\,z$ and Im$\,z$ represents its real and imaginary parts.  Also, $\bar{z}$ denotes the complex conjugate of $z$.We denote by $\mathbf{z}\in\C^l$ the vector $(z_1,...,z_l)$, where $z_m=x_m+iy_m$ with $x_m,z_m$ the imaginary parts of $z_m$. We set $\big\bracevert\! \mathbf{z}\!\big\bracevert$ for the vector $(|z_1|,\ldots,|z_l|)$. This is not to be confused with $|\mathbf{z}|=\sqrt{|z_1|^2+\cdots+|z_l|^2}$ which denotes the usual norm of the vector $\mathbf{z}\in\mathbb{C}^l$. As usual,  the operators $\partial/\partial z_m$ and $\partial/\partial \widebar{z}_m$ are defined by
$$
\frac{\partial}{\partial z_m}=\frac{1}{2}\left(\frac{\partial}{\partial x_m}-i\frac{\partial}{\partial} y_m\right),\quad \frac{\partial}{\partial \widebar{z}_m}=\frac{1}{2}\left(\frac{\partial}{\partial x_m}+i\frac{\partial}{\partial} y_m\right)
$$

We denote the stantard Sobolev, the homogeneous Sobolev and the Lebesgue spaces by  $H^{s,p}=H^{s,p}(\mathbb{R}^d)$, $\dot{H}^{s,p}=\dot{H}^{s,p}(\R^d)$ and $L^p=L^p(\mathbb{R}^d)$, respectively, with its usual norms. We denote $H^s=H^{s,2}$ and $\dot{H}^s=\dot{H}^{s,2}$. Given a time interval $I$, the mixed Lebesgue space $L_t^pL_x^q(I\times\mathbb{R}^d)$ is  denoted by $L_t^pL_x^q$ and will be endowed with the norm
$$
\Vert f \Vert_{L_t^pL_x^q}=\left(\int_I\left(\int_{\mathbb{R}^d}|f(t,x)|^qdx\right)^{p/q}dt\right)^{1/p},
$$
with the obvious modification if either $p=\infty$ or $q=\infty$. Also, we will denote by $p'$ the exponent conjugate of $p$, that is, $1=1/p+1/p'$.

The Fourier transform of a function defined on $\R^d$ will be usually denoted by $\hat{f}$.
For $s\in\mathbb{R}$, we define the fractional differentiation/integral operator
$
\widehat{|\nabla|^sf}(\xi):=|\xi|^s\hat{f}(\xi),
$
which defines the homogeneous Sobolev norm:
$$
\Vert f\Vert_{\dot{H}_x^s(\mathbb{R}^d)}:=\Vert |\nabla|^sf\Vert_{L_x^2(\mathbb{R}^d)}.
$$

For the measure part of the work we set the following notation. Let $X$ be a locally compact Hausdorff space. We denote by $\mathcal{C}_b(X)$ the space of all bounded continuous functions on $X$, and by $\mathcal{C}_c(X)$ the space of all continuous functions with compact support on $X$. Moreover, $\mathcal{M}_+(X)$ denotes the Banach space of all non-negative measures on $X$, $\mathcal{M}_+^b(X)$ the space of all finite (or bounded) measures and $\mathcal{M}_+^1$ the space of all probability measures. Given two measures $\nu$ and $\mu$, we write $\nu \ll \mu$ to indicate that $\nu$ is absolutely continuous with respect to $\mu$. For any $\mu \in \mathcal{M}_+^b(X)$, the quantity $|\mu| := \mu(X)$ is called the mass of $\mu$. Let us introduce some notions of convergence in the context of measures.
\begin{definition}
	$\bf{}$
	\begin{itemize}
		\item[(i)] A sequence $(\mu_m)\subset\mathcal{M}_+$ is said to converge vaguely to $\mu$ in $\mathcal{M}_+(X)$, denoted by $\mu_m\overset{\ast}{\rightharpoonup}\mu$, if $\int_Xfd\mu_m\rightarrow\int_Xfd\mu$ for all $f\in\mathcal{C}_c(X)$.
		
		\item[(ii)] A sequence $(\mu_m)\subset\mathcal{M}_+^b(X)$ is said to converge weakly to $\mu$ in $\mathcal{M}_+^b(X)$, denoted by $\mu_m\rightharpoonup\mu$, if $\int_X fd\mu_m\rightarrow\int_X fd\mu$, for all $f\in \mathcal{C}_b(X)$.
		
		\item[(iii)] A sequence $(\mu_m)\subset\mathcal{M}_+^b$ is said to be uniformly tight if, for every $\epsilon>0$, there exists a compact subset $K_\epsilon\subset X$ such that $\mu_m(X\backslash K_\epsilon)\leq \epsilon$ for all $m$. We also say that a set $\mathcal{H}\subset \mathcal{M}_+(X)$ is vaguely bounded if $\sup_{\mu\in\mathcal{H}}\left|\int_Xf d\mu\right|<\infty$ for all $f\in\mathcal{C}_c(X)$.
	\end{itemize}
\end{definition}

If no confusion is caused, we denote $\int_{\mathbb{R}^d}f(x)dx$ simply by $\int f$.

Next, let us recall the Sobolev inequalities.
\begin{lemma}[Sobolev Embedding]\label{embedding} Let $s\in(0,+\infty)$ and $1\leq p<\infty$. 
	\begin{itemize}
		\item[(i)] If $s\in (0,\frac{d}{p})$ then $H^{s,p}(\R^d)$ is continuously embedded in $L^r(\R^d)$ where $s=\frac{d}{p}-\frac{d}{r}$. Also, it holds 
		\begin{equation}\label{sobolev}
			\Vert f\Vert_{L^r}\lesssim \Vert |\nabla|^s f\Vert_{L^p}.
		\end{equation}
		\item[(ii)] If $s=\frac{d}{2}$ then $H^s(\R^d)\subset L^r(\R^d)$ for all $r\in[2,\infty)$. In particular,
		\begin{equation}\label{sobolev2}
			\Vert f\Vert_{L^r}\lesssim \Vert f\Vert_{H^s}.
		\end{equation}
		\item[(iii)] If $s>\frac{d}{2}$ then $H^s(\R^d)\subset L^\infty(\R^d)$.
	\end{itemize}
	
\end{lemma}
\begin{proof}
See \cite[Proposition 4.18]{demengel}.
\end{proof}

\begin{lemma}[Sobolev multiplication law]\label{sobolevmult}. Let $d\geq 1$ and assume that $s,s_1,s_2$ are real number satisfying either

		\begin{itemize}
			\item[(i)] $s_1+s_2\geq 0$, $s_1,s_2\geq s$ and $s_1+s_2>s+d/2$; or
			\item[(ii)] $s_1+s_2> 0$, $s_1,s_2> s$ and $s_1+s_2\geq s+d/2.$
		\end{itemize}
	Then, there is a continuous multiplication map $H^{s_1}(\R^d)\times H^{s_2}(\R^d)\rightarrow H^s(\R^d)$ given by $(u,v)\mapsto uv$. Moreover, holds
	$$
	\Vert uv\Vert_{H^s(\R^d)}\leq C\Vert u\Vert_{H^{s_1}(\R^d)}\Vert v\Vert_{H^{s_2}(\R^d)}.
	$$
	\end{lemma}
	\begin{proof}
		See \cite[Corollary 316]{tao4}.
	\end{proof}

The first result we present here about the biharmonic Schrödinger equation is the linear estimates, that are directly consequences of the dispersion estimates presented in \cite{AKS}. The result reads as follows

\begin{lemma}[Linear estimates]\label{linear estimates} Let $2\leq p\leq\infty$ and $g\in L^{p'}$. For all time $t\neq 0$,  we have
	$$
	\Vert e^{it\Delta^2}g\Vert_{L^p}\lesssim |t|^{-\frac{d}{4}\left(1-\frac{2}{p}\right)}\Vert g\Vert_{L^{p'}}.
	$$
	\end{lemma}

The above result is a key ingredient to derive Strichartz type estimates associated to the biharmonic Schrödinger propagator presented in \cite{Pausader}. Before states the result we need to set some concepts. We say that the pair $(q,r)$ is biharmonic Schrödinger admissible (B-admissible for short) if satisfies
\begin{equation}\label{badmim}
	\frac{4}{q}=\frac{d}{2}-\frac{d}{r},\quad 2\leq q,r\leq \infty,\,(q,r,d)\neq (2,\infty,4).
\end{equation}
and a pair $(q,r)$ is said to be Schrödinger admissible (S-admissible for short) if satisfies
\begin{equation}\label{sadmim}
\frac{2}{q}=\frac{d}{2}-\frac{d}{r},\quad 2\leq q,r\leq \infty,\,(q,r,d)\neq (2,\infty,2).
\end{equation}
The result states the following.
\begin{proposition}[Strichartz-type estimates]\label{strichartz}
	Let $I=[0,T]$ be an interval, $h\in (L^1_{\rm loc})_tH_x^{-4}(I\times\R^d)$ and $u\in C_tH_x^{-4}(I\times\R^d)$ be a solution of
$$
u(t)=e^{it\Delta^2}u_0-\int_0^t  e^{i(t-s)\Delta^2}h(s)ds,
$$
where $u_0\in L^2$. 
\begin{itemize}
	\item[(i)] For any B-admissible pairs $(q,r)$ and $(a,b)$ 
	\begin{equation}\label{S310}
		\Vert u\Vert_{L_t^qL_x^r}\lesssim \Vert u_0\Vert_{L^2}+\Vert h\Vert_{L_t^{a'}L_x^{b'}}.
	\end{equation}
	\item[(ii)] For any S-admissible pairs $(q,r)$ and $(a,b)$ and $s\geq 0$,
	\begin{equation}\label{S311}
		\Vert |\nabla|^s u\Vert_{L_t^qL_x^r}\lesssim \Vert |\nabla|^{s-\frac{2}{q}}u_0+\Vert |\nabla|^{s-\frac{2}{q}-\frac{2}{a}}h\Vert_{L_t^{a'}L_x^{b'}}.
	\end{equation}
	\item[(iii)] For any B-admissible pair $(q,r)$ and $u_0\in\dot{H}^2(\R^d)$,
	\begin{equation}\label{S319}
		\Vert \Delta u\Vert_{L_t^qL_x^r}\lesssim \Vert \Delta u_0\Vert_{L^2}+\Vert \nabla h\Vert_{L_t^2L_x^{\frac{2d}{d+2}}}
	\end{equation}
	\end{itemize}

\end{proposition}
\begin{proof}
	For (i) and (ii) see \cite[Proposition 3.1]{Pausader}. Item (iii) follows from item (ii) by noticing that from Sobolev's embedding we have 
	$$
	\Vert |\nabla|^s u\Vert_{L_t^1L_x^r}\lesssim \Vert |\nabla|^{s+\frac{2}{q}} u\Vert_{L_t^qL_x^{\tilde{r}}},
	$$
	where $\tilde{r}$ is such that $\frac{2}{q}=\frac{d}{\tilde{r}}-\frac{d}{r}$, together with the fact that if $(q,r)$ is B-admissible then $(q,\tilde{r})$ is S-admissible. 
	\end{proof}
\begin{remark}
 The above results still hold with if we replace $e^{it\Delta^2}$ with $U_k(t)$, $\kl$ (see Definition \ref{solution}).
\end{remark}

\begin{definition}\label{QPMS}(Almost periodicity modulo symmetries).
	A solution $\mathbf{u}$ to \eqref{SISTB} on a time interval $I$ is said to be almost periodic modulo symmetries if there exist functions $N:I\longrightarrow \mathbb{R}^+$, $x:I\longrightarrow \mathbb{R}^d$ and $C:\mathbb{R}^+\longrightarrow\mathbb{R}^+$, such that for all $t\in I$ and $\eta>0$:
	$$
	\sum_{k=1}^l\int_{|x-x(t)|\geq C(\eta)/N(t)}\gamma_k|\Delta u_k(t,x)|^2dx\leq \eta
	$$
	and
	$$
	\sum_{k=1}^l\int_{|\xi|\geq C(\eta)N(t)}\gamma_k|\xi|^4|\hat{u}_k(t,\xi)|^2d\xi\leq \eta.
	$$ 
	The functions $N$ is called scale frequency function of the solution $\mathbf{u}$, $x$ is the spacial center function and $C$ is the compactness modulus function.
\end{definition}

\begin{remark}\label{almost periodic}
	Recall that a set $\mathcal{F}\subset\dot{H}^2(\mathbb{R}^d)$ is compact if and only if $\mathcal{F}$ is norm-bounded in $\dot{H}_x^2(\mathbb{R}^d)$ and, for all $\eta>0$, there exists a compactness modulus function $C(\eta)>0,$ such that
	$$
	\int_{|x|\geq C(\eta)}|\Delta f(x)|^2dx+\int_{|\xi|\geq C(\eta)}|\xi|^4|\hat{f}(\xi)|^2d\xi\leq \eta
	$$
	for all functions $f\in\mathcal{F}$ (this essentially follows from the well-known Kolmogorov-Riesz theorem combined with Plancherel's theorem). Also, by Sobolev embedding, every compact set in $\dot{H}_x^2(\R^d)$ is compact in $L_x^{\frac{2d}{d-4}}(\R^d)$. In particular, a solution of \eqref{SISTB} that is almost periodic modulo symmetries  satisfies
	$$
	\sum_{k=1}^l\int_{|x-x(t)|\geq C(\eta)/N(t)}|u_k(t,x)|^{\frac{2d}{d-4}}\lesssim \eta,
	$$
	for all $t\in I$ and $\eta>0$. 
\end{remark}

The next result is a useful tool in order to prove coercivity results. 
\begin{lemma}\label{lema comp}
	Let $I\subset\mathbb{R}$ be an open interval with $0\in I$, $a\in\R$, $b>0$ and $q>1$. Define $\gamma=(bq)^{-1/(q-1)}$ and $f(r)=a-r+br^q$, for $r>0$. Let $G(t)$ be a nonnegative continuous function such that $f\circ G\geq 0$ in $I$. Assume that $a<(1-\delta)\displaystyle\left(1-\displaystyle\frac{1}{q}\displaystyle\right)\gamma$, for some $\delta>0$ sufficiently small, we have
	\begin{itemize}
		\item[(i)] If $G(0)<\gamma$ then there exists $\delta_1=\delta_1(\delta)>0$ such that $G(t)<(1-\delta_1)\gamma$, for all $t\in I$;
		\item[(ii)] If $G(0)>\gamma$ then there exists $\delta_2=\delta_2(\delta)$ such that $G(t)>(1+\delta_2)\gamma$, for all $t\in I$.
	\end{itemize}
\end{lemma}
\begin{proof}
	See Corollary 3.2 in \cite{pastor2}.
\end{proof}
\begin{remark}
	The result also holds with $\delta=0=\delta_1=\delta_2$. For details see \cite[Lemma 3.1]{pastor2}.
\end{remark}
\section{Consequences of the assumptions H's} This section is devoted to show some properties that are consequences of our assumptions \ref{H1}-\ref{H7}. The reader can check the details  in \cite[Chapter 2]{NoPa2} with the suitable modifications when necessary. We start with a lemma that characterizes  the nonlinearities $f_k$.

\begin{lemma}\label{consequences} Let $5\leq d\leq 16$. Suppose that \ref{H1} and \ref{H2} hold. Then
	\begin{itemize}
		\item[(i)] For any $\mathbf{z},\mathbf{z}'\in \mathbb{C}^l$, 
		\begin{equation}\label{fk21}
			|f_k(\mathbf{z})-f_k(\mathbf{z}')|\lesssim \sum_{m=1}^l\sum_{j=1}^l\left(|z_j|^{\frac{8}{d-4}}+|z'_j|^{\frac{8}{d-4}}\right)|z_m-z'_m|,\quad k=1,...,l.
		\end{equation}
		In particular, 
		$$
		|f_k(\mathbf{z})|\lesssim \sum_{m=1}^l|z_m|^{\frac{d+4}{d-4}}.
		$$
		\item[(ii)] Let $\mathbf{u}$ and $\mathbf{v}$ be complex-valued functions defined on $\R^d$.  Then, for $\kl$, $5\leq d< 12$,
		\begin{equation}
			\left|\nabla[f_k(\mathbf{u})-f_k(\mathbf{v})]\right|\lesssim \sum_{m=1}^l\sum_{j=1}^l|u_j|^{\frac{8}{d-4}}|\nabla u_m-\nabla v_m|+\sum_{m=1}^l\sum_{j=1}^l\left(|u_j|^{\frac{12-d}{d-4}}+|v_j|^{\frac{12-d}{d-4}}\right)|u_j-v_j||\nabla v_m|.
		\end{equation}
						and for $12\leq d\leq 16$,
			\begin{equation}\label{gradine}
				\left|\nabla[f_k(\mathbf{u})-f_k(\mathbf{v})]\right|\lesssim \sum_{m=1}^l\sum_{j=1}^l|u_j|^{\frac{8}{d-4}}|\nabla u_m-\nabla v_m|+\sum_{m=1}^l\sum_{j=1}^l|u_j-v_j|^{\frac{8}{d-4}}|\nabla v_m|, \quad 
				\end{equation}
		\item[(iii)] Let $1<p,q,r<\infty$ be such that $\displaystyle\frac{1}{r}=\frac{1}{p}\left(\frac{8}{d-4}\right)+\frac{1}{q}$. Then, for $k=1,\ldots,l$, $5\leq d< 12$,
		\[
		\Vert \nabla [f_k(\mathbf{u})-f_k(\mathbf{v})]\Vert_{\mathbf{L}^r}\lesssim \Vert\mathbf{u}\Vert_{\mathbf{L}^p}^{\frac{8}{d-4}}\Vert\nabla (\mathbf{u}-\mathbf{v})\Vert_{\mathbf{L}^q}+\left(\Vert \ub\Vert_{L^p}^{\frac{12-d}{d-4}}+\Vert \vb\Vert_{L^p}^{\frac{12-d}{d-4}}\right)\Vert \mathbf{u}-\mathbf{v}\Vert_{\mathbf{L}^p}\Vert\nabla\mathbf{v}\Vert_{\mathbf{L}^q}.
		\]
		and for $12\leq d\leq 16$,
		$$
		\Vert \nabla [f_k(\mathbf{u})-f_k(\mathbf{v})]\Vert_{\mathbf{L}^r}\lesssim \Vert\mathbf{u}\Vert_{\mathbf{L}^p}^{\frac{8}{d-4}}\Vert\nabla (\mathbf{u}-\mathbf{v})\Vert_{\mathbf{L}^q}+\Vert \mathbf{u}-\mathbf{v}\Vert_{\mathbf{L}^p}^{\frac{8}{d-4}}\Vert\nabla\mathbf{v}\Vert_{\mathbf{L}^q},
		$$
		In particular, taking $\mathbf{v}=0$ we have
		$$
		\Vert \nabla f_k(\mathbf{u})\Vert_{\mathbf{L}^r}\lesssim \Vert\mathbf{u}\Vert_{\mathbf{L}^p}^{\frac{8}{d-4}}\Vert\nabla \mathbf{u}\Vert_{\mathbf{L}^q}
		$$
	\end{itemize}
\end{lemma}
\begin{proof}
	Part $(i)$ is consequence of the Fundamental Theorem of Calculus. Part $(ii)$ is proved by using the chain rule and \cite[Remark 2.3]{cazetal}, noticing that for $5\leq d\leq 16$ we have $\frac{8}{d-4}\leq 1$ and $\frac{8}{d-4}>1$ for $d>12$. Part $(iii)$ is obtained by combining part $(ii)$ and Hölder's inequality. 
\end{proof}
\begin{lemma}\label{gradsum}
 Assume that \ref{H2} holds. For all $J\in\N$ we have for $\kl$, $5\leq d< 12$,
	\[
	\left| \nabla\left[\sum_{j=1}^J f_k(\mathbf{u}_j)-f_k\left(\sum_{j=1}^J \mathbf{u}_j\right)\right]\right|\lesssim \sum_{j\neq i}\left( |\nabla \mathbf{u}_j||\mathbf{u}_i|^{\frac{8}{d-4}}+ |\ub_j||\ub_i|^{\frac{12-d}{d-4}}|\nabla \ub_i |\right) 
	\]
	and for $12\leq d\leq 16$,
	\[
	\left| \nabla\left[\sum_{j=1}^J f_k(\mathbf{u}_j)-f_k\left(\sum_{j=1}^J \mathbf{u}_j\right)\right]\right|\lesssim \sum_{j\neq i} |\nabla \mathbf{u}_j||\mathbf{u}_i|^{\frac{8}{d-4}}
	\]
\end{lemma}
\begin{proof}
	The proof is a straightforward calculation using the chain rule 
	$$
	\frac{\partial f_k}{\partial x_j}(\mathbf{u})=\sum_{m=1}^l\left(\frac{\partial f_k}{\partial z_m}(\mathbf{u})\frac{\partial u_m}{\partial x_j}+\frac{\partial f_k}{\partial \bar{z}_m}(\mathbf{u})\frac{\partial \bar{u}_m}{\partial x_j}\right),
	$$
	where $\mathbf{u}_j=(u_{j1},...,u_{jl})$ and $f_k(\mathbf{z})=f_k(z_1,...,z_l)$, the assumption \ref{H2} and Lemma \ref{consequences}.
\end{proof}

The next result gives us a characterization for the pontential function $F$. 
\begin{lemma}\label{lemma22}
	Assume that \ref{H1}, \ref{H2}, \ref{H3}, \ref{H5} and \ref{H7} hold. 
	\begin{itemize}
		\item[(i)] Let $\mathbf{z},\mathbf{z}'\in\C^l$. Then
		\begin{equation}\label{fk22}
			|\RE F(z)-\RE F(\mathbf{z}')|\lesssim \sum_{m=1}^l\sum_{j=1}^l\left(|z_j|^{\frac{d+4}{d-4}}+|z'_j|^{\frac{d+4}{d-4}}\right)|z_m-z_m'|.
		\end{equation}
		In particular, 
		\begin{equation}\label{fk23}
			|\RE F(\mathbf{z})|\lesssim \sum_{j=1}^l|z_j|^{\frac{2d}{d-4}}.
		\end{equation}
		\item[(ii)] Let $\mathbf{u}$ be a complex-valued function defined on $\R^d$. Then
		\begin{equation}\label{fk24}
			\RE \sum_{k=1}^lf_k(\ub)\widebar{u}_k=\left(\frac{2d}{d-4} \right)\RE F(\ub).
		\end{equation}
		\item[(iii)] Let $\ub:\R^d\rightarrow \C^l$. Then,
		\begin{equation}\label{regradF}
			\RE\sumk f_k(\ub)\nabla\bar{u}_k=\RE[\nabla F(\ub)].
		\end{equation}
		
		\item[(iv)] We have
		\begin{equation}\label{fk25}
			f_k(x)=\frac{\partial F}{\partial x_k}(\mathbf{x}),\quad \mathbf{x}\in\R^l.
		\end{equation}
		In addition, $F$ is positive on the positive cone of $\R^l$.
	\end{itemize}
\end{lemma}
\begin{proof}
	To prove part (i) we first note that we can write
	\begin{equation}\label{refk}
		f_k(\zb)=2\frac{\partial}{\partial \widebar{z}_k}\RE F(\zb).
	\end{equation}
	Combining with \eqref{fk21} and using the Fundamental Theorem of Calculus we get \eqref{fk22}. The inequality \eqref{fk23} is the particular case of \eqref{fk22} when $\zb=\mathbf{0}$. For part (ii) we differentiate both sides of \ref{H5} with respect to $\lambda$ and evaluate at $\lambda=1$ and then take the real part and use \ref{H3}. For part (iii), we first differentiate $F$ with respect to $x_j$, apply the chain rule. Then, we take the real part of the result ad use \ref{H3} to get the result. The last part is consequence of \ref{H7}, \eqref{refk} and \eqref{fk25}.
\end{proof}

	Next lemma show us that the structure of our system is Gauge invariant. 
\begin{lemma}\label{GC}
	Assume that \ref{H3} and \ref{H4} hold. For any $\theta\in\R$ and $\zb\in\C^l$, we have
	\begin{itemize}
		\item[(i)] $\RE F\displaystyle\left(e^{i\frac{\sigma_1}{2}\theta}z_1,...,e^{i\frac{\sigma_l}{2}\theta}z_l\right)=\RE F(\zb).$\\
		\item[(ii)] The nonlinearities $f_k$, $k=1,...,l$ satisfy the Gauge condition
		$$
		f_k\left( e^{i\frac{\sigma_1}{2}\theta}z_1,...,e^{i\frac{\sigma_l}{2}\theta}z_l\right)= e^{i\frac{\sigma_k}{2}\theta}z_kf_k(\zb).
		$$
	\end{itemize}
\end{lemma}
\begin{proof}
	Part (i) is a consequence of the fact that we can write
	$$
	f_k(\zb)=2\frac{\partial}{\partial\widebar{z}_k}\RE F(\zb)
	$$
	and the chain rule. Part (ii) is consequence of part (i) and once again the chain rule. 
\end{proof}

Next lemma is very helpful in order to get conservation laws.

\begin{lemma}\label{MRequality} Assume that \ref{H3} and \ref{H4} hold. Then, for all $\zb\in\C^l$,
	$$ \IM \sumk \frac{\alpha_k}{\gamma_{k}}f_k(\zb)\bar{z}_k=0.   $$
	\end{lemma}
\begin{proof}
Let $\zb\in\C^l$. For $\theta\in\R$, denote $(w_1,...,w_l):=\left(e^{i\frac{\alpha_1}{\gamma_1}}z_1,...,e^{i\frac{\alpha_l}{\gamma_l}}z_l\right)$. Thus by applying Lemma \ref{GC} we have that $f_k(\wb)\bar{w}_k=f_k(\zb)\bar{z}_k$ for all $\kl$. Now, set $h(\theta)=F(\wb)$. After apply the chain rule, take the real part in order to get 
	$$
	\RE \frac{dh}{d\theta}=\IM\sumk\left(\frac{\alpha_k}{\gamma_k}\right)f_k(\zb)\bar{z}_k.
	$$
	Moreover, taking the derivative with respect to $\theta$ on both sides of \ref{H4}, we get that $\RE\frac{dh}{d\theta}=0$ and the result follows. 
\end{proof}
The last consequence of the assumptions is the following lemma.
 \begin{lemma}\label{fkcontinua}
Let $1\leq d\leq 16$. Suppose that \ref{H1} and \ref{H2} hold. Then, for all $\kl$ we have $f_k\in C(\mathbf{H}^2(\R^d),\mathbf{H}^{-4}(\R^d))$.
\end{lemma}
\begin{proof}
	Let $\ub_i\in\mathbf{H}^2$ be such that $\ub_i\rightarrow \ub$ in $\mathbf{H}^2$, as $i\rightarrow\infty$. In particular, $\Vert \ub_i\Vert_{\mathbf{H}^2}\leq M$, for some $M>0$. By Lemma \ref{consequences} and Lemma \ref{sobolevmult} we have, for $\kl$,
	\begin{equation*}
		\begin{split}
			\Vert f_k(\ub_i)-f_k(\ub)\Vert_{H^{-4}}&\lesssim \sumj \summ \Vert \left(|u_{ij}|+|u_j|\right)|u_{im}-u_m|\Vert_{H^{-4}}\\
			&\lesssim \sumj \summ \Vert |u_{ij}||u_{mi}-u_m|\Vert_{H^{-4}}+\sumj \summ \Vert |u_j||u_{im}-u_m|\Vert_{H^{-4}}\\
			&\lesssim \sumj \summ \Vert u_{ij}\Vert_{H^2}\Vert u_{im}-u_m\Vert_{H^2}+\sumj\summ \Vert_{u_j}\Vert_{H^2}\Vert u_{im}-u_m\Vert_{H^2}\\
			&\lesssim M\summ \Vert u_{im}-u_m\Vert_{H^2}+\sumj\summ \Vert_{u_j}\Vert_{H^2}\Vert u_{im}-u_m\Vert_{H^2}.\\
		\end{split}
	\end{equation*}
	By the convergence of $(\ub_i)$ we see that the right-hand side goes to $0$ as $i\rightarrow\infty$. Hence, $f_k(\ub_i)\rightarrow f_k(\ub)$ in $\Hb^{-4}(\R^d)$ and the result follows. 
\end{proof}
 
 \begin{remark}\label{dimension}
 	While the  above result allows us to apply Proposition \ref{strichartz}, it also imposes a restriction on the dimension we can work with. To address this issue, alternative tools must be employed. 
 \end{remark}

\section{Local well-posedness} 

In this section we will develop the local theory for \eqref{SISTB} in the energy critical case. To begin with, we will set some norms that will be useful along the work. For $I\subset\R$, we set the following norms
\begin{equation}\label{norms}
	\begin{split}
		\Vert f \Vert_{M(I)} &:= \Vert \Delta f \Vert_{L_t^\frac{2(d+4)}{d-4}L_x^{\frac{2d(d+4)}{d^2+16}}(I\times\R^d)}  \\
		\Vert f \Vert_{W(I)} &:= \Vert \nabla f \Vert_{L_t^\frac{2(d+4)}{d-4}L_x^{\frac{2d(d+4)}{d^2-2d+8}}(I\times\R^d)}  \\
		\Vert f \Vert_{Z(I)} &:= \Vert  f \Vert_{L_{t,x}^\frac{2(d+4)}{d-4}(I\times\R^d)} \\
		\Vert f \Vert_{N(I)} &:= \Vert \nabla f\Vert_{L_t^2L_x^{\frac{2d}{d+2}}(I\times\R^d)}  \\
		\Vert f \Vert_{B(I)} &:= \Vert  f \Vert_{L_{t,x}^\frac{2(d+4)}{d}(I\times\R^d)}  \\
	\end{split}
\end{equation}

\begin{remark}\label{estimatespaces}
Let us establish here some estimates that will be used several times during the section. We start noticing that for $\ub\in\Dz$, by Sobolev's inequality we have
$$
\Vert \ub \Vert_{\Zb(I)} \lesssim \Vert \ub \Vert_{\Wb(I)}\lesssim\Vert \ub \Vert_{\Mb(I)}.
$$
Moreover, by using Lemma \ref{consequences}, one can easily see that 
$$
\Vert f_k(\ub)\Vert_{N(I)}\lesssim \Vert \ub\Vert_{\Zb(I)}^{\frac{8}{d-4}}\Vert \ub\Vert_{\Wb(I)}\lesssim \Vert \ub\Vert_{\Zb(I)}^{\frac{8}{d-4}}\Vert \ub\Vert_{\Mb(I)}.
$$
Another useful consequence of Lemma \ref{consequences} is that if $\ub\in\Wb(I)$ then $f_k(\ub)\in N(I)$. Also, observe that the pairs in the norm $M(I)$ and $B(I)$ are B-admissible and the pair in the norm $N(I)$ is a exponent conjugate of an S-admissible pair, namely $(2,2d/(d-2))$. Then, by Strichartz estimate \eqref{S319}, we have
$$
\Vert \ub\Vert_{\Mb (I)}\lesssim \Vert \Delta \ub_0\Vert_{\Lb^2}+\Vert f_k(\ub)\Vert_{N(I)},
$$ 
which means that we can control the second derivative of $\ub$ using only one derivative of $f_k$. One can also notice that, if $I_t=[0,t]\subset I$ by using the above estimates we arrive at $\Vert \ub\Vert_{\Mb(I_t)}\leq  C\left( \Vert\Delta\ub_0\Vert_{\Lb^2}+\Vert\ub\Vert_{\Mb(I_t)}^{\frac{d+4}{d-4}}\right)$. Then, provided that all involved norms are finite, one can use a continuity argument to get
$$
\Vert \ub\Vert_{\Mb(I)}\lesssim  \Vert\Delta\ub_0\Vert_{\Lb^2}.
$$

\end{remark}

The local existence result is a straight forward adaptation of the standard contraction map argument presented in \cite[Chapter 4]{cazenave}, so we will omit the proof. The result reads as follows.
\begin{theorem}[Local existence]\label{localexi}	Let $5\leq d\leq 16$. There exist $\delta>0$ such that for any initial data $\ub_0\in\mathbf{H}^2$, and a time interval $I=[0,T]$. If 
	\begin{equation}\label{P52}
		\Vert U_k(t)u_{k0}\Vert_{W(I)}\leq\delta,\quad \kl,
		\end{equation}
		  then there exists a unique solution $\ub\in\mathbf{C}_t\Hb^2(I\times\R^d)$ such that $\ub(0)=\ub_0$ and $\ub\in \Mb(I)\cap \Bb(I)$. Moreover, the following estimates hold
		  \begin{equation}\label{P53}
		  		  	\Vert \ub\Vert_{\Wb(I)}	\leq 2\delta,\hbox{ and}\quad 
		  	\Vert \ub\Vert_{\Mb(I)}+\Vert \ub\Vert_{\Lb_t^\infty \Hb_x^2}\leq C\left(\Vert \ub_0\Vert_{\Hb^2}+\delta^{\frac{d+4}{d-4}}\right).
		  	\end{equation}
		\end{theorem}

\begin{corollary}[Small data global existence]\label{smalldata}
	Let $5\leq d\leq 16$. There exists $\delta_0>0$ such that for any $\ub_0\in\Dz$ obeying $\Vert \Delta \ub_0\Vert_{\Lb^2}\leq \delta_0$, the solution $\ub$ for \eqref{SISTB} obtained in Theorem \ref{localexi} extends globally and satisfies
	\begin{equation}\label{S1R}
		\Vert \Delta \ub\Vert_{\Mb(\R)}\lesssim \Vert \Delta\ub_0\Vert_{\Lb^2},
	\end{equation}
\end{corollary}
		\begin{proof} 
		Note that by using Sobolev's inequality and Strichartz \eqref{S319} we have for $\kl$ and any compact interval $I$ 
	$$
	\Vert U_k(t)u_k(t_0)\Vert_{W(I)}\lesssim \Vert U_k(t)u_k(t_0)\Vert_{M(I)}\lesssim \Vert u_k(t_0) \Vert_{\dot{H}^2}\leq \delta_0,
	$$
	 Therefore, by Proposition \ref{localexi}, the solution exists in any compact $I$ and, hence, is global. To prove \eqref{S1R} notice that for compact interval $I$, 
	 \begin{equation*}
	 	\begin{split}
	 			 	\Vert \Delta \ub\Vert_{\Mb(I)}&\leq \Vert \Delta \ub_0\Vert_{\Lb^2}+C\Vert f_k(\ub)\Vert_{\Nb(I)}\\
	 			 	&\leq \Vert \Delta \ub_0\Vert_{\Lb^2}+C\Vert  \ub\Vert_{\Mb(I)}^{\frac{d+4}{d-4}}\\
	 	\end{split}
	 \end{equation*} 
	 where we have used Strichartz \eqref{S319} and Lemma \ref{consequences}. Since these norms are finite, a continuity argument gives that
	 $$
	 \Vert  \ub\Vert_{\Mb(I)}\lesssim \Vert\Delta \ub_0\Vert_{\Lb^2}
	 $$
	 for all compact interval $I\subset \R$, which proves \eqref{S1R}.
\end{proof}

\begin{proposition}[Global existence criterion]\label{globalexicrit}
	Let $5\leq d\leq 16$. Let $\ub\in \Cb_t\Hb^2_x(I\times \R^d)$, with  $I=[0,T)$, be a solution of \eqref{SISTB} such that $S_I(\ub)=Z<\infty$. Then, there exists $K=K(\Vert \ub_0\Vert_{\Hb^2}, Z)$ such that
	\begin{equation}\label{P57}
		\Vert \ub\Vert_{\Bb(I)}+\Vert \ub\Vert_{\Lb_t^\infty\dot{\Hb}^2(I\times\R^d)}+\Vert\ub\Vert_{\Mb(I)}\leq K,
	\end{equation}
	and $\ub$ can be extend to a solution $\tilde{\ub}\in\Cb_t\Hb_x^2(J\times\R^d)$ of \eqref{SISTB}, for some $J=[0,T')$ with $T'>T$.
\end{proposition}
\begin{proof}
The proof follows by the same strategy as in \cite[Proposition 5.2]{Pausader} with the suitable modification for our problem and applying Lemma \ref{consequences}. So we omit the details.
\end{proof}
As a direct consequence we have the following finite time blow-up criterion. 
\begin{corollary}[Standard finite time blow-up criterion]\label{standblowup} Let $5\leq d\leq 16$ and $\mathbf{u}$ be the maximal-lifespan solution defined on interval $I=[0, T^*)$ with initial data $\mathbf{u}(0)=\mathbf{u}_0$. If $T^*<\infty$, then $S_I(\ub)=\infty$.
	\end{corollary}
	
The continuous dependence of initial data is a consequence of the following stability result. This type of result was introduced in \cite{tao} for the NLS equation and has a lot of importance in physical models. Stability results in the fourth-order Schrödinger equation was introduced by \cite{Pausader}. Here we will only state the vectorial version of the result given in \cite{Pausader}. Since the proof is analogous only taking into account that we are working in vectorial case, it will be omitted. One can also check \cite[Proposition 4.5]{FH} for a proof in the NLS system regime. The result states the following.

\begin{proposition}\label{stability} Let $5\leq d\leq 16$ and $I\subset\R$ be a compact interval containing $0$. Let $\vb:I\times\R^d\rightarrow\C^l$ be an approximate solution of \eqref{SISTB} in the sense that 
	$$
	i\alpha_k\partial_tv_k+\gamma_{k}\Delta^2v_k+\varepsilon\Delta v_k- f_k(\vb)=e_k,
	$$
	for some $\mathbf{e}=(e_1,...,e_l)\in\Nb(I)$. Assume also that $\Vert \vb\Vert_{\Zb(I)}=Z<\infty$ and $\Vert\vb\Vert_{\Lb_t^\infty\dot{\Hb}_x^2(I\times\R^d)}=L<\infty$. For any $\eta>0$ there exists $\delta_0>0$, $\delta_0=\delta_0(\eta, Z, L)$ , such that if $\Vert \mathbf{e}\Vert_{\Nb(I)}\leq \delta$, and $\ub_0\in\Hb^2$ satisfies
	\begin{equation}\label{P62}
		\Vert \vb(0)-\ub_0\Vert_{\Dz}\leq \eta\hbox{ and } \left\Vert U_k(t)\left(v_k(0)-u_{k0} \right)\right\Vert_{W(I)}\leq \delta\quad\kl
	\end{equation}
	for some $\delta\in(0,\delta_0]$, then there exists $\ub\in \Cb_t\Hb_x^2(I\times\R^d)$ a solution of \eqref{SISTB} such that $\ub(0)=\ub_0$. Moreover, $\ub$ satisfies
	\begin{equation}\label{P63}
		\begin{split}
			&\Vert \ub-\vb\Vert_{\Wb(I)}\leq C\left(\delta + \delta^{\frac{15}{(d-4)^2}}\right),\\
			&\Vert \ub-\vb\Vert_{\Lb_t^q\dot{\Hb}_x^{2,r}(I\times\R^d)}\leq C\left(\eta+\delta+\delta^{\frac{15}{(d-4)^2}}\right),\hbox{ and}\\
			&\Vert \ub\Vert_{\Lb_t^q\dot{\Hb}_x^{2,r}(I\times\R^d)}\leq C,
		\end{split}
	\end{equation}
	for all B-admissible pairs $(q,r)$, where $C=C(\eta, Z, L)$.
\end{proposition}

Notice that if we let $e_k=0$ for $\kl$, then Proposition \ref{stability} wields the uniform continuity for the solution map $\ub_0\mapsto \ub$. Indeed, in this case we have that $\vb$ is a solution to \eqref{SISTB} with initial data $\vb_0$. Moreover from Strichartz and Remark \ref{estimatespaces}, one can see that
$$
\left\Vert U_k(t)\left(v_k(0)-u_k(0) \right)\right\Vert_{W(I)}\leq C\Vert v_k(0)-u_k(0)\Vert_{\dot{H}^2}, \quad\kl.
$$
Thus, if we choose $\delta=C\eta$, we can make the respectively solutions close enough by controlling $\delta$ when $\eta$ is small.

\section{Ground state solutions}
This section is devoted to prove existence of ground state solution. We start by proving that the functionals $K$ and $P$ are $C^1$ functions. This allow us to see solutions to \eqref{eliptic} as critical points of the energy functional, which give sense to the Definition \ref{weaksolution}. 

\begin{lemma}\label{frechet} The functionals $K$ and $P$ are in $C^1(\Dz;\R)$. In addition, if $\gb\in\dz$, then
	\begin{equation}\label{k'}
		K'(\psib)(\gb)=2\sumk\gamma_k\int\Delta\psi_kg_kdx,
	\end{equation}
	and
	\begin{equation}\label{p'}
		P'(\psib)(\gb)=\sumk\int f_k(\psib)g_kdx.
	\end{equation}
\end{lemma}
\begin{proof} The proof of \eqref{k'} is similar to those presented in \cite[Exemple 1.3(d)]{AP}. One may also check \cite[Examples 1.3.13]{BS}. 
	
	To prove \eqref{p'}, we start showing that $P\in C^1(\dz,\R)$. Indeed, let $(\ub_m)\subset\dz$ be a sequence such that $\ub_m\rightarrow\ub$ in $\dz$ as $m\rightarrow \infty$. Therefore, Lemma \ref{lemma22}, Hölder's inequality and Sobolev embedding yields
	\begin{equation}\label{n319}
		\begin{split}
			|P(\ub_m)-P(\ub)|&\leq \int |F(\ub_m)-F(\ub)|dz\\
			&\leq C\sumk \sumj \int\left(|u_{jm}|^{\frac{d+4}{d-4}}+|u_j|^{\frac{d+4}{d-4}}\right)|u_{km}-u_k|dx\\
			&\leq C\sumk \sumj \left(\Vert u_{jm}\Vert_{L^{\frac{2d}{d-4}}}^{\frac{d+4}{d-4}}+\Vert u_{j}\Vert_{L^{\frac{2d}{d-4}}}^{\frac{d+4}{d-4}} \right)\Vert u_{km}-u_k\Vert_{L^{\frac{2d}{d-4}}}\\
			&\leq C\sumk \sumj \left(\Vert u_{jm}\Vert_{H^2}^{\frac{d+4}{d-4}}+\Vert u_{j}\Vert_{H^2}^{\frac{d+4}{d-4}} \right)\Vert u_{km}-u_k\Vert_{H^2}.\\
		\end{split}
	\end{equation}
	Since the right-hand side goes to zero as $m\rightarrow\infty$, the claim follows. Left us to prove that $P$ is Gatêaux differentiable and its derivate is given by \eqref{p'}. Indeed, let $\psib\in\dz$ and $g\in C_0^\infty(\R^d)^l$. If $0<t\leq 1$, by Lemma \ref{consequences} and Lemma \ref{lemma22} we have
	\begin{equation*}
		\begin{split}
			\frac{1}{t}\left[F(\psib+t\gb)\right.&-\left.F(\psib)-t\sumk f_k(\psib)g_k\right]\leq\frac{1}{t}\left[|F(\psib+t\gb)-F(\psib)|+t\sumk|f_k(\psib)||g_k|\right]\\
			&\leq \frac{C}{t}\left[\summ\sumj \left(|\psi_j+tg_j|^{\frac{d+4}{d-4}}+|\psi_j|^{\frac{d+4}{d-4}}\right)t|g_m|+t\sumk\sumj |\psi_j|^{\frac{d+4}{d-4}}|g_k|\right]\\
			&\leq C\left[ \summ\sumj \left( |\psi_j|^{\frac{d+4}{d-4}}+t^{\frac{d+4}{d-4}}|g_j|^{\frac{d+4}{d-4}} \right)|g_m|+\sumk\sumj|\psi_j|^{\frac{d+4}{d-4}}|g_k| \right]\\
			&\leq C\left[ \summ\sumj |g_j|^{\frac{d+4}{d-4}}|g_m|+\sumk\sumj |\psi_j|^{\frac{d+4}{d-4}}|g_k| \right]
		\end{split}
	\end{equation*}
	where we used that $t\leq 1$ in the last inequality. Also, a similar argument used in \eqref{n319} and the Sobolev embedding $H^2\hookrightarrow L^{\frac{2d}{d-4}}$ shows that the right-hand side is a $L^1$ function. On the other hand, the chain rule and \eqref{fk25} gives us 
	$$
	\lim_{t\rightarrow 0}\frac{1}{t}\left[F(\psib+t\gb)-F(\psib)-t\sumk f_k(\psib)g_k\right]=0,
	$$
	which, together with Dominated Convergence, implies that
	\[
	\begin{split}
		P'(\psib)(\gb)-\int\sumk f_k(\psib)g_kdx &= \lim_{t\rightarrow 0} \frac{1}{t}\left[ P(\psib+t\gb)-P(\psib)-t\int\sumk f_k(\psib)g_kdx \right]\\
		&=\lim_{t\rightarrow 0}\frac{1}{t}\int \left[F(\psib+t\gb)-F(\psib)-t\sumk f_k(\psib)g_k\right]\\
		&= \int \lim_{t\rightarrow 0}\frac{1}{t}\left[F(\psib+t\gb)-F(\psib)-t\sumk f_k(\psib)g_k\right]\\
		&=0.
	\end{split}
	\]
	Therefore, $P$ is Gatêaux differentiable. To finish the proof, notice that Lemma \ref{fkcontinua} implies that $P'$ is continuous and, consequently, its Gatêaux and Fréchet derivatives coincide. 
	\end{proof}
	
	\begin{remark}
		Note the one may take $\gb\in\dz$ in the previous proof by using a density argument.
	\end{remark}
	
Once the definition of ground state is well established, we can now turn our attention to prove the existence. We start by observing that given $\psib$ a solution of \eqref{eliptic}, if we take $g_k=\psi_k$ in \eqref{weaksolution}, summing over $k$ and using Lemma \ref{lemma22},  we have
\begin{equation}\label{Ppositive}
	K(\psib)=\sumk \int f_k(\psib)\psi_kdx=\left(\frac{2d}{d-4}\right)P(\psib),
\end{equation}
which gives us the following relations between the energy functional $E$, the kinetic energy $K$ and the potential energy $P$,
 \begin{equation}\label{pohozaevident}
 P(\psib)=\left(\frac{d-4}{4}\right)E(\psib)\quad\hbox{and} \quad K(\psib)=\left(\frac{d}{2}\right)E(\psib).
 \end{equation}
Moreover, one can also notice that \eqref{Ppositive} implies that if $\psib\in\mathcal{C}$ then $P(\psib)>0$, that is, the set of all non trivial solutions lies in  $\mathcal{P}:=\{\psib\in\mathcal{C};\, P(\psib)>0\}$. Moreover, by using Lemma \ref{lemma22} and Sobolev's inequality we have
\begin{equation}\label{sobolevtype}
	P(\psib)\leq  CK(\psib)^{\frac{d}{d-4}}.
\end{equation}
Thus, if we set the functional $J(\psib)$ by
$$
J(\psib):=\frac{K(\psib)^{\frac{d}{d-4}}}{P(\psib)},
$$
for all $\psib\in\mathcal{P}$, by \eqref{Ppositive} and \eqref{pohozaevident}, 
\begin{equation}\label{SJ}
	J(\psib)=\left(\frac{d-4}{2d}\right)\left(\frac{d}{2}\right)^{\frac{4}{d-4}}E(\psib)^{\frac{4}{d-4}}:=C_dE(\psib)^{\frac{4}{d-4}}.
\end{equation}
Hence, a non-trivial solution of \eqref{eliptic} is a ground state if, and only if, minimizes $J$. Therefore, in order to prove existence of ground state solution, we have to find the best constant in \eqref{sobolevtype}, that is, the best $C>0$ such that
$$
\frac{1}{C}\leq J(\psib),\quad  \psib\in\Pc.
$$
The optimal constant to place is given by
\begin{equation}\label{Min}
	C^{-1}_{\rm opt}:=\inf\{J(\psib); \psib\in\Pc\}.
\end{equation}
In particular, by \eqref{SJ} and the definition of energy \eqref{energy},  a non-trivial solution of \eqref{eliptic} ia a ground state if and only if  has least energy among all non-trivial solution of \eqref{eliptic}. 
\begin{remark}\label{minimization} By using the homogeneity of the functionals $K$ and $P$ one may notice the following:
	\begin{itemize}
		\item[(i)] A straightforward computation shows that $K(\bbc\ub\bbc)\leq K(\ub)$ and together with \ref{H6} yields that $J(\bbc\ub\bbc)\leq J(\ub)$. Thus, there is no loss of generality in assuming that minimizing sequences are always non-negative.
		
		\item[(ii)] 	$C_{\rm opt}^{-1}=I^{\frac{d}{d-4}}.$ Therfore, to prove that the infimum \eqref{Min} is attained we can consider the normalized problem given by
		\begin{equation}\label{MinNorm}
			I:=\inf\{ K(\ub);\, \ub\in\Pc,\, P(\ub)=1\}
		\end{equation}
		and solve \eqref{MinNorm} instead of \eqref{Min}.
		
		\item[(iii)] The functionals $K$ and $P$ are invariant under the transformation
		\begin{equation}\label{obs34ii}
			\ub\mapsto\ub^{R,y}=R^{-\frac{d-4}{2}}\ub\left( R^{-1}(x-y)\right),\quad R>0,y\in\R^d.
		\end{equation}

	\end{itemize}	

\end{remark}
Before solve the minimization problem, we set a minimizing sequence for \eqref{Min} to be a sequence $\ub_m$ in $\mathcal{P}$ such that $J(\ub_m)\rightarrow	C_{\rm opt}^{-1}$. In the same way, a minimizing sequence for \eqref{MinNorm} is a sequence $\ub_m$ in $\Pc$ such that $P(\ub_m)=1$ for all $m$. 
\subsection{Concentration-compactness}
In order to solve \eqref{MinNorm}, we will apply the concentration compactness method introduced in \cite{lions2}. This is a well known method, so we will present here just a outline of the proof. For a detailed version in the NLS regime, one may check \cite[Theorem 3.11]{pastor3}. See also \cite[Theorem 3.12]{HP} and \cite[Theorem 3.8]{FH} . The result reads as follows
\begin{theorem}\label{lionsI1}
	Suppose that $(\ub_m)$ is a minimizing sequence for \eqref{MinNorm}. Then, up to translation and dilation $(\ub_m)$ is relatively compact in $\mathcal{P}$, that is, there exist a subsequence $(\ub_{m_j})$ and sequences $(R_j)\subset\mathbb{R}$, $(y_j)\subset \mathbb{R}^d$ such that  $\vb_j$ given by
	$$
	\vb_j:=R_j^{-\frac{d-4}{2}}\ub_{m_j}(R_j^{-1}(x-y_j)),
	$$
	strongly converges in $\mathcal{N}$ to some $\vb$, which minimizes \eqref{MinNorm}.
\end{theorem}
\begin{proof}[Outline of the proof]
	
The proof proceeds in several steps. First, we will need a result that is an adapted version of Lemma 1.7.4 in \cite{cazenave} concerning the concentration function given by
\begin{equation}\label{concfunc}
	Q_m(R):=\sup_{y\in\mathbb{R}^d}\int_{B(y,R)}F(\ub_m)dx,\quad R>0,
\end{equation}
so we will omit the details. The results reads as follows.
\begin{lemma}\label{lema310}
	Let $(\ub_m)\subset{\mathbf{L}}^{\frac{2d}{d-4}}(\mathbb{R}^d)$ be such that $\ub_m\geq \mathbf{0}$ and $\int F(\ub_m)dx=1$, for any $m\in\mathbb{N}$. Let $Q_m(R)$ be the concentration function of $F(\ub_m)$ defined by \eqref{concfunc}. Then, for each $m\in\N$, there is $y=y(m,R)$ such that
	$$
	Q_m(R)=\int_{B(y,R)}F(\ub_m)dx. 
	$$
\end{lemma} 
The next step is to built a suitable sequence of probability measures. We do this by considering $(\ub_m)$ to be any minimizing sequence of \eqref{MinNorm}, that is, 
\begin{equation}\label{339}
	\lim_{m\rightarrow\infty}K(\ub_m)=I\quad\hbox{and}\quad P(\ub_m)=\int F(\ub_m)dx=1.			
\end{equation} 
	Thus, by using the invariance \eqref{obs34ii} and Lemma \eqref{lema310}	we get sequences $(R_j)\subset\R$ and $(y_j)\subset\R^d$ such that
	\begin{equation*}
		\vb_j:=R_j^{-\frac{d-4}{2}}\ub_{m_j}(R_j^{-1}(x-y_j)),
	\end{equation*}
	satisfies
	\begin{equation}\label{341}
		\sup_{y\in\mathbb{R}^d}\int_{B(y,1)}\varphi(\vb_m)dx=\int_{B(0,1)}\varphi(\vb_m)dx=\frac{1}{2}.
	\end{equation}
	Once again we use the invariance \eqref{obs34ii} and conclude that $(\vb_m)$ is also a minimizing sequence for \eqref{MinNorm}. In particular, $(\vb_m)$ is uniformly bounded in $\mathcal{P}$. Hence, up to a subsequence, there exist $\vb\in\Dz$ such that
	\begin{equation}\label{345}
		(\vb_m)\rightharpoonup \vb\quad \mbox{in }  \Dz.
	\end{equation} 
	Next, we set the sequences of measures $(\mu_m)$ and $(\nu_m)$ by
	$$
	\mu_m:=\sum_{k=1}^l\gamma_k|\Delta v_{km}|^2dx,\quad\hbox{and}\quad \nu_m:=F(\vb_m) dx.
	$$
	Since $(\vb_m)$ is a minimizing sequence of \eqref{MinNorm}, then $(\nu_m)$ is a probability sequence of measures. This allow us to use the following auxiliary result, which is inspired in \cite[Lemma I.1]{lions1}.
\begin{lemma}[Concentration-compactness lemma I] \label{lema35}	
	Suppose that $(\nu_m)$ is a sequence in $\mathcal{M}_+^1(\mathbb{R}^d)$. Then, there is a subsequence, still denoted by $(\nu_m)$, such that one of the following conditions holds:
	\begin{itemize}
		\item[(i)](Vanishing) For all $R>0$ it holds
		$$
		\lim_{m\rightarrow\infty}\left(\sup_{x\in\mathbb{R}^d}\nu_m(B(x,R))\right)=0.
		$$
		
		\item[(ii)](Dichotomy) There is a number $\lambda\in(0,1)$ such that for all $\epsilon>0$ there exists $R>0$ and a sequence $(x_m)$ with the following property: given $R'>R$
		$$
		\nu_m(B(x_m,R))\geq \lambda-\epsilon,
		$$
		$$
		\nu_m(\mathbb{R}^d\backslash B(x_m,R'))\geq 1-\lambda-\epsilon,
		$$
		for $m$ sufficiently large.
		
		\item[(iii)](Compactness) There exists a sequence $(x_m)\subset\mathbb{R}^d$ such that for each $\epsilon>0$ there is a radius $R>0$ with the property
		$$
		\nu_m(B(x_m,R))\geq 1-\epsilon,
		$$
		for all $m$.
	\end{itemize}
\end{lemma}
\begin{proof}
	See \cite[Lemma 23]{FM}.
\end{proof}
As a consequence, $\nu_m$ must satisfy one of the three above conditions. The next step is show that $\nu_m$ cannot satisfy vanishing and dichotomy and, thus, implying the compactness of sequence. The vanishing is ruled out immediately by \eqref{341}. To eliminate the possibility of dichotomy, we need to construct a localized Sobolev-type inequality. We do this by adapting Lemma 8 in \cite{FM} to our settings. The result is the following.
\begin{lemma}\label{lemma38}
	For every $\delta>0$ there is a constant $C(\delta)>0$ satisfying the following: if $0<r<R$ with $r/R\leq C(\delta)$ and $x\in\R^d$, then there is a cutt-off function $\chi_R^r\in H^{2,\infty}(\R^d)$ such that $\chi_R^r=1$ on $B(x,r)$, $\chi_R^r=0$ outside $B(x,R)$,
	\begin{equation}\label{334}
		K(\chir \ub)\leq \sumk \gamma_k \int_{B(x,R)}|\Delta u_k|^2dy+\delta K(\ub),	
	\end{equation}
	and
	\begin{equation}\label{335}
		K((1-\chir)\ub)\leq\sumk\int_{\R^d\backslash B(x,r)}|\Delta u_k|^2dy+\delta K(\ub),
	\end{equation}
	for any $\ub\in\Dz$.
\end{lemma}
\begin{proof}
	The proof follows the ideas in \cite[Lemma 8]{FM}.There is no loss of generality in assuming that $x=0$, so we will assume that. Let $\chir$ be given by
	\begin{equation}
		\chir(y)=\begin{cases}
			1, &|y|\leq r,\\
			\displaystyle\frac{\log(|y|/R)}{\log(r/R)}, & r\leq |y|\leq R,\\
			0, &|y|\geq R.\\
		\end{cases}
	\end{equation}
	A straightforward computation yields that $\chir\in H^{2,\infty}$ and
	\begin{equation}\label{336}
		\int_{B(0,R)}|\nabla \chir|^d dy=\frac{\omega_d}{(\log(R/r))^{d-1}},\quad \int_{B(0,R)}|\Delta\chir|^{\frac{d}{2}}dy=\frac{(d-2)^{\frac{d}{2}}\omega_d}{\log(R/r)^{\frac{d-2}{2}}},
	\end{equation}
	where $\omega_d$ is the measure of the unit sphere in $\R^d$. Now, by using \eqref{336}, Sobolev, Young and Hölder's inequalities, for any $\epsilon>0$ and $\kl$,
	\begin{equation*}
		\begin{split}
			\int_{B(0,R)}|\Delta[\chir u_k]|^2 dy 	&\leq
			(1+\epsilon)\int_{B(0,R)}|\Delta u_k|^2dy+\left(1+\frac{1}{\epsilon}\right)\frac{C(d-2)^{2}\omega_d^{\frac{4}{d}}}{\log(R/r)^{\frac{2(d-2)}{d}}}\int |\Delta u_k|^2 dy\\
			&\quad +\left(1+\epsilon+\frac{1}{\epsilon}\right)\frac{C\omega_d^{\frac{2}{d}}}{(\log(R/r))^{\frac{2(d-1)}{d}}}\int |\Delta u_k|^2dy.
		\end{split}
	\end{equation*}
	Multiplying by $\gamma_k$ and summing over $k$, we arrive at
	\begin{equation*}
		\begin{split}
			K(\chir\ub)&\leq \sumk\gamma_k\int_{B(0,R)}|\Delta u_k|^2dy\\
			&\quad+\left[\epsilon+\left(1+\frac{1}{\epsilon}\right)\frac{\zeta_1^2}{\log(R/r)^{\frac{2(d-2)}{d}}}+\left(1+\epsilon+\frac{1}{\epsilon}\right)\frac{\zeta_2^2}{(\log(R/r))^{\frac{2(d-1)}{d}}}\right]K(\ub),
		\end{split}
	\end{equation*}
	
	where $\zeta_1=\sqrt{C}(d-2)\omega_d^{\frac{2}{d}}$ and $\zeta_2=\sqrt{C}\omega_d^{\frac{1}{d}}$. So, for any $\delta>0$, one may find a $C(\delta)$ such that if $r/R\leq C(\delta)$ than
	$$
	\left[\epsilon+\left(1+\frac{1}{\epsilon}\right)\frac{\zeta_1^2}{\log(R/r)^{\frac{2(d-2)}{d}}}+\left(1+\epsilon+\frac{1}{\epsilon}\right)\frac{\zeta_2^2}{(\log(R/r))^{\frac{2(d-1)}{d}}}\right]\leq \delta
	$$
	and \eqref{334} follows.  Also, \eqref{335} follows as in \eqref{334}, since $|\Delta (1-\chir)|^2=|\Delta \chir|^2$, $|\nabla(1-\chir)|^2=|\nabla \chir|^2$ and $\chir=0$ outside $B(0,R)$ and 
	\begin{equation*}
		\begin{split}
			\int_{\R^d\backslash B(0,r)}|\Delta[(1-\chir)u_k]|^2dy &\leq (1+\epsilon)\int_{\R^d\backslash B(0,r)}|1-\chir|^2|\Delta u_k|^2dy\\
			&\quad + \left(1+\frac{1}{\epsilon}\right)\int_{\R^d\backslash B(0,r)}|u_k|^2|\Delta(1-\chir)|^2dy\\
			&\quad +\left(1+\epsilon+\frac{1}{\epsilon}\right)\int_{\R^d\backslash B(0,r)}|\nabla(1-\chir)|^2|\nabla u_k|^2dy.\\
		\end{split}
	\end{equation*}
\end{proof}
With this in hand, we state the localized Sobolev-type inequality as follows. The proof is the same as in \cite[Corollary 9]{FM}, so the details will be omitted.
\begin{lemma}\label{localsob}
	Let $\ub\in\Dz$ with $\ub\geq\mathbf{0}$. Fix $\delta>0$ and $r/R\leq C(\delta)$ as in Lemma \eqref{lemma38}. Then,
	\begin{equation}\label{337}
		\int_{B(x,r)}F(\ub)dy\leq I^{-\frac{d}{d-4}}\left[\int_{B(x,R)}\sumk\gamma_k |\Delta u_k|^2dy+\delta K(\ub)\right]^{\frac{d}{d-4}},
	\end{equation}
	\begin{equation}\label{338}
		\int_{\R^d\backslash B(x,R)}F(\ub)dy\leq I^{\frac{-d}{d-4}}\left[ \int_{\R^d\backslash B(x,r)}\sumk\gamma_k |\Delta u_k|^2dy+(2\delta+\delta^2) K(\ub)\right]^{\frac{d}{d-4}}.
	\end{equation}
\end{lemma}
	Now, we proceed by contradiction and invoke Corollary \ref{localsob} to preclude the dichotomy case. Therefore, Lemma \eqref{lema35} yield the existence of a sequence $(x_m)\in\R^d$ such that for all $\epsilon>0$, there is a radius $R>0$ with
\begin{equation}\label{350}
	\nu_m(B(x_m))\geq 1-\epsilon,\quad\forall m,
\end{equation}
which together with \eqref{341} and the fact that $(\nu_m)$ is a sequence of probability measures, we get that $(\nu_m)$ is a uniformly tight sequence. Therefore, up to a subsequence,   $(\nu_m)$ converges weakly to some $\nu\in\Mmu$ (see for instance \cite[Theorems 30.6 and 31.2]{bauer}), that is, 
\begin{equation}\label{351}
	\int fd\nu_m\rightarrow \int fd\nu,\quad\forall f\in C_b(\R^d)
\end{equation}

Now, notice that $(\mu_m)$ is vaguely bounded because $(K(\vb_m))$ is uniformly bounded. Then, up to a subsequence, there is $\mu\in\Mmb$ such that
\begin{equation}\label{353}
	\mu_m\cfe\mu\quad\hbox{ in }\Mmb.		
\end{equation}
In particular, $\mu(\R^d)\leq\liminf_{m\rightarrow\infty}\mu_m(\R^d)$.

The next step is to use the ``second concentration compactness lemma''. Such result is the last key ingredient to prove Theorem \ref{lionsI1}, and is an adapted version of the limit case lemma presented in \cite[Lemma I.1]{lions2}. The proof follows the same approach as in the NLS regime presented in \cite[Lemma 6]{NoPa3}, regarding the differences for the BNLS regime, so we will omit the details. The result reads as follows.
\begin{lemma}[Concentration-compactness lemma II] \label{lema36}
	Let $(\ub_m) \subset\Dz$ be a sequence such that $\ub_m\geq \mathbf{0}$ and 
	\begin{equation}\label{314}
		\left\{\begin{array}{lcc}
			(\ub_m)\rightharpoonup \ub, & \hbox{in}& \Dz,\\
			\mu_m:=\displaystyle\sum_{k=1}^l\gamma_k|\Delta u_{km}|^2dx\overset{\ast}{\rightharpoonup} \mu & \hbox{in}& \mathcal{M}_+^b(\mathbb{R}^d)\\
			\nu_m:=F(\ub ) dx \overset{\ast}{\rightharpoonup} \nu,& \hbox{in}& \mathcal{M}_+^b(\mathbb{R}^d). \\
		\end{array}\right.
	\end{equation}
	Then,
	\begin{itemize}
		\item[(i)] There exists an at  most countable set $J$, a family of distinct points $\{x_j\in \mathbb{R}^d;\, j\in J\}$, and a family of non-negative numbers $\{\nu_j;\, j\in J\}$ such that
		\begin{equation}\label{315}
			\nu=F(\ub)dx+\sum_{j\in J}\nu_j\delta_{x_j}.
		\end{equation}
		
		\item[(ii)]Moreover, we have
		\begin{equation}\label{316}
			\mu \geq  \sum_{k=1}^l\gamma_k|\Delta u_k|^2dx +\sum_{j\in J}\mu_j\delta_{x_j}
		\end{equation}
		fore some family $\{\mu_j;\, j\in J\},b_j>0,$ such that
		\begin{equation}\label{317}
			\nu_j\leq I^{-\frac{d}{d-4}}\mu_j^{\frac{d}{d-4}},\quad \forall j\in J.
		\end{equation}
		In particular, $\displaystyle\sum_{j\in J}\nu_j^{\frac{d-4}{d}}<\infty$.
	\end{itemize}
\end{lemma}
	Next, \eqref{345}, \eqref{351} and \eqref{353} allow us to use Lemma \ref{lema36} to obtain
	\begin{equation}\label{354}
		\mu\geq\sumk\gamma_k|\Delta v_k|^2dx+\sum_{j\in J}\mu_j\delta_{x_j}\quad\hbox{and}\quad \nu=F(\vb)dx+\sum_{j\in J}\nu_j\delta_{x_j},			
	\end{equation}
	for some family $\{x_j\in\R^d;\, j\in J\}$, with $J$ countable and $\mu_j, \nu_j$ non-negatives numbers satisfying
	\begin{equation}\label{355}
		\nu_j\leq I^{-\frac{d}{d-4}}\mu_j^{\frac{d}{d-4}},\quad\forall j\in J,
	\end{equation}
	with $\sum_{j\in J}\nu_j^{\frac{d-4}{d}}$ convergent. Therefore, since $\nu(\R^d)=1$, \eqref{sobolevtype} and \eqref{355} yields
	\begin{equation}\label{356}
		\begin{split}
			I=\liminf_{m\rightarrow\infty}\mu_m(\R^d)&\geq\mu(\R^d)\geq K(\vb)+\sum_{j\in J}\mu_j\geq I\left[P(\vb)^{\frac{d-4}{d}}+\sum_{j\in J}\nu_j^{\frac{d-4}{d}}\right]\\
			&\geq I\left[P(\vb)+\sum_{j\in J}\nu_j\right]^{\frac{d-4}{d}}=I[\nu(\R^d)]=I,
		\end{split}
	\end{equation} 
	where we also used that $\lambda\mapsto\lambda^{\frac{d-4}{d}}$ is a strictly concave function. Since $P(\vb)\neq 0$, we observe that $\nu_j=0$ must hold for all $j\in J$, otherwise, if $\nu_{j_0}\neq 0$, for some $j_0$, then \eqref{354} would yield that $\nu=\nu_{j_0}\delta_{x_{j0}}$ and hence
	\begin{equation}\label{357}
		1=\nu(\mathbb{R}^d)=\nu_{j_0}.
	\end{equation}
However, the normalization \eqref{341} leads to
	$$
	\frac{1}{2}\geq\lim_{m\rightarrow\infty}\nu_m(B(x_{j_0},1))=\nu(B(x_{j_0},1))=\int_{B(x_{j_0},1)}d\nu=\nu_{j_0},
	$$
	which contradicts  \eqref{357}. As a result, we deduce that $\nu=F(\vb)dx$. Moreover, since $\nu\in\Mmu$, we have that $\vb\in\mathcal{P}$.
	
	Finally, notice that by the definition of $I$, we know that $I\leq K(\vb)$. On the other hand, the lower semi-continuity of the weak convergence \eqref{345}	yields $K(\vb)\leq\liminf_{m\rightarrow\infty}K(\vb_m)=I$. Hence, $K(\vb)=I=\lim_{m\rightarrow\infty}K(\vb_m)$ and, moreover, $\vb_m\rightarrow\vb$ strongly in $\Dz$, as desired.
\end{proof}
 Note that actually we have proved the following:
 
 \begin{corollary}\label{minimum}
 	There exists $\vb\in\mathcal{P}$ such that $P(\vb)=1$ and $K(\vb)=C_{\rm opt}^{\frac{d-4}{d}}$, where $C_{\rm opt}$ is the best constant in the general critical Sobolev-type inequality \eqref{sobolevtype}.
 \end{corollary}

We are now in a position to prove Theorem \ref{ESGS}. 

\begin{proof}[Proof of Theorem \ref{ESGS}] Let $\vb$ be the minimizer of \eqref{MinNorm}, that is, $K(\vb)=I$ and $P(\vb)=1$. So, by Remark \eqref{minimization}, we have
	\[
	J(\vb)=\frac{K(\vb)^{\frac{d}{d-4}}}{P(\vb)}=I^{\frac{d}{d-4}}=C_{\rm opt}^{-1}.
	\]
	Therefore, $\vb$ is also a minimizer of $J$. By \eqref{SJ}, $\vb$ is a minimum of $E$ and hence is a critical point of $E$, that is,  for any $\gb\in\dz$, $E'(\vb)(\gb)=0$. 	By Lemma \ref{frechet}, follows that $\vb$ is a solution to \eqref{eliptic}. Consequently $\vb$ is a ground state solution, as we required.
	
\end{proof}
	
	\begin{remark}
	One may notice that we do not have any information about the uniqueness of the ground state solutions for \eqref{eliptic}. Despite that,  Remark \ref{minimization} (ii) implies that $C_{\rm opt}^{-1}$ is indeed the optimal constant to \eqref{sobolevtype} and do not depend on the choice of the ground state, that is, for all $\ub\in\mathcal{P}$, holds
	\begin{equation}\label{optimal}
		P(\ub)\leq C_{\rm opt}K(\ub)^{\frac{d}{d-4}},
	\end{equation}
	where the optimal constant is given by
	\begin{equation}\label{optcte}
		C_{\rm opt}=\frac{1}{C_d}\frac{1}{E(\psib)^{\frac{4}{d-4}}}=\frac{1}{C_d}\left(\frac{d}{2K(\psib)}\right)^{\frac{4}{d-4}},\quad \hbox{where }C_d=\left(\frac{d-4}{2d}\right)\left(\frac{d}{2}\right)^{\frac{4}{d-4}}.
	\end{equation}
	\end{remark}
	
	\section{Blow-up}
	As mentioned before, this section is devoted to prove that solutions of \eqref{SISTB} with positive initial energy below that of the ground states but
	with kinetic energy above that of the ground states, or negative energy must blow up in finite time. Our strategy rely on the argument introduced in \cite{BL17}. The first step is to prove a localized virial identity. Our first result is the following
	\begin{lemma}\label{IV}(Virial identity) Let 	$\phi$ to be a smooth function and consider $\phi_R=R^2\phi(x/R)$ for $R>0$. Set 
		\[
		\V_R(t)=-2\sumk\alpha_k\IM\int\nabla\phi_R\cdot\nabla u_k\bu_kdx.
		\]
		Then
		\begin{equation}
			\begin{split}
				\V_R'(t)&= 8\RE\sumk\gk\sum_{j,n,m}\int\partial^2_{n,m}\phi_R\partial^2_{m,j}u_k\partial^2_{j,n}\bu_kdx\\
				&\quad-4\RE\sumk \gk\sum_{j,n}\int\partial_j\bu_k\partial^2_{j,n}\Delta\phi_R\partial_nu_kdx\\
				&\quad-2\sumk\gk\int\Delta^2\phi_R|\nabla u_k|^2dx+\sumk\gk\int\Delta^3\phi_R|u_k|^2dx\\
				&\quad-\frac{16}{d-4}\RE\int F(\ub)\Delta \phi_R.\\
			\end{split}
		\end{equation}
	\end{lemma}
	\begin{proof}
		We start by using that $2i \IM(z)=z-\bar{z}$ to get
		\begin{equation}\label{R1}
			i\V_R(t)=-2i\IM\sumk \int\ak \bu_k\nabla u_k\cdot\nabla \phi_Rdx
			=\sumk\int\ak u_k\nabla \bu_k\cdot\nabla \phi_Rdx-\sumk \int \ak \bu_k\nabla u_k\cdot \phi_Rdx.
		\end{equation}
		Using integration by parts we have
		\begin{equation}\label{R2}
			\int\sumk\ak|u_k|^2\Delta\phi_R=-\sumk\int\ak u_k\nabla \bu_k\cdot\nabla\phi_Rdx-\sumk\int\ak \bu_k\nabla u_k\cdot\nabla \phi_Rdx.
		\end{equation}
		Combining \eqref{R1} and \eqref{R2} we get
		\begin{equation}\label{R3}
			\begin{split}
				i	\V_R(t)&=-2\sumk\int\ak\bu_k\nabla u_k\cdot\nabla\phi_Rdx-\sumk\int\ak\bu_k\Delta\phi_R u_kdx\\
				&=-\sumk\int\ak \bu_k[\nabla\phi_R\cdot\nabla +\Delta\phi_R]u_kdx-\sumk\int\ak \bu_k\nabla u_k\cdot\nabla\phi_Rdx.
			\end{split}	
		\end{equation}
		By the definition of the Morawetz potential we conclude that
		\[
		\V_R(t)=-2\RE\sumk\int\ak\bu_k\Gamma_{\phi_R}u_kdx,
		\]
		where $\GR:=-i\left(\nabla\phi_R\cdot\nabla +\Delta\phi_R\right)$. By introducing the notation
		\[ \V_R(t):= -\sumk\langle\ak u_k,\GR u_k\rangle=-\sumk\langle i\ak u_k, i\GR u_k\rangle,  \]
		taking the time derivative and using \eqref{SISTB}, we obtain
		\[\begin{split}
			\V'_R(t)&=\sumk\gk\langle u_k,\left[\Delta^2, i\GR \right]u_k \rangle - \sumk\langle u_k, \left[f_k(\ub), i\GR\right] u_k \rangle\\
			&:= A-B,
		\end{split}\]
		where $[X,Y ]=XY-YX$. Note that the term $A$ is essentially the same as $\mathcal{A}_R^{[1]}$ in \cite[Lemma 3.1]{BL17}, taking into account the suitable modifications for the vectorial case. Then we get
		\begin{equation}\label{AA}
			\begin{split}
				A&= 8\RE\sumk\gk\sum_{j,n,m}\int\partial^2_{n,m}\phi_R\partial^2_{m,j}u_k\partial^2_{j,n}\bu_kdx\\
				&\quad-4\RE\sumk \gk\sum_{j,n}\int\partial_j\bu_k\partial^2_{j,n}\Delta\phi_R\partial_nu_kdx\\
				&\quad-2\sumk\gk\int\Delta^2\phi_R|\nabla u_k|^2dx+\sumk\gk\int\Delta^3\phi_R|u_k|^2dx.\\
			\end{split}
		\end{equation}
		For $B$, we use integration by parts and Lemma \ref{lemma22} to get
		
		\begin{equation}\label{B1}
			\begin{split}
				B&= 2\sumk\RE\left[\int f_k(\ub)i\GR\bu_kdx-\int\bu_k i\GR f_k(\ub)dx\right]\\
				&=2\sumk\RE\left[ \int f_k(\ub)\nabla\phi_R\cdot\nabla\bu_kdx+\int f_k(\ub)\bu_k \Delta\phi_Rdx \right]\\
				&\quad -2\sumk\RE\left[ \int \bu_k\nabla\phi_R\cdot\nabla f_k(\ub)dx +\int f_k(\ub)\Delta\phi_R\bu_kdx\right]\\
				&=2\sumk\RE\left[2\int f_k(\ub)\nabla u_k\cdot\nabla\phi_Rdx+\int f_k(\ub)\bu_k\Delta\phi_Rdx\right]\\
				&=4\RE\int\nabla F(\ub)\cdot\nabla\phi_Rdx+\frac{4d}{d-4}\RE\int F(\ub)\Delta \phi_Rdx\\
				&=\frac{16}{d-4}\RE\int F(\ub)\Delta \phi_Rdx.
			\end{split}
		\end{equation}
		The proof is complete by summing \eqref{AA} and \eqref{B1}.
	\end{proof}
	Now, consider $\phi:\R^d\rightarrow\R$ to be a radial function with regularity property $\nabla^j\phi\in L^\infty(\R^d)$ for $1\leq j\leq6$ and such that 
	\[  \phi(r)=\begin{cases}
		r^2/2,&r\leq 1\\
		\hbox{constant},&r\geq 10\\
	\end{cases}\quad\hbox{and}\quad \phi''(r)\leq 1,\hbox{ for } r\geq 0,\]
	and we define, for $R>0$, the rescaled function  $\phi_R:\R^d\rightarrow\R$ by
	\[\phi_R(r):=R^2\phi\left(\frac{r}{R}\right).\]
    A straightforward calculation gives us that for all $r\geq 0$,
	\begin{equation}\label{PHI}
		\phi_R''(r)\geq 1,\quad 1-\frac{\phi_R'(r)}{r}\geq 0, \quad \Delta\phi_R(r)\leq d.
	\end{equation}
	Then, we have the following result.
	\begin{lemma}\label{virialBL}
		Let $5\leq d\leq 16$. 	Suppose that $\ub$ is radial solution of \eqref{SISTB} defined on the interval $[0,T)$.
		
		Then for any $t\in I$, we have
		\begin{equation}\label{virialinequality}
			\V'_R(t)\leq \frac{16d}{d-4}E(\ub_0)-\frac{32}{d-4}K(\ub)+O\left(R^{-4}+R^{-\frac{4(d-1)}{d-4}}\sumk\Vert \nabla u_k\Vert_{L^2}^{\frac{4}{d-4}}+R^{-2}\sumk \Vert \nabla u_k\Vert_{L^2}^2\right).
		\end{equation}
	\end{lemma}
	\begin{proof}
		Let $\ub$ be a radial solution of \eqref{SISTB}.  Then, using integration by parts we have the bound (see \cite[Section 3.1]{BL17} for details)
		\begin{equation}\label{v1}
			8\langle u_k, \partial_{j,n}^2(\partial^2_{n,m}\phi_R)\partial_{m,j}u_k  \rangle\leq 8\int|\Delta u_k|^2,\quad\kl.
		\end{equation}
		Also, by definition of $\phi_R$, one can easily check that, for $\kl$,
		\begin{equation}\label{v2}
			\begin{split}
				|\langle u_k, \partial_j(\partial^2_{j,n}\Delta \phi_R)\partial_nu_k\rangle|&\lesssim R^{-2}\Vert\nabla u_k\Vert_{L^2}^2\\
				|\langle u_k,\partial_j(\Delta^2\phi_R)\partial_l   \rangle| &\lesssim R^{-2}\Vert\nabla u_k\Vert_{L^2}^2\\
				|\langle u_k,\Delta^3 \phi_R u_k\rangle|&\lesssim R^{-4}\Vert u_k\Vert_{L^2}^2
			\end{split}
		\end{equation}
		Therefore, combining \eqref{v1}, \eqref{v2} with Lemma \ref{virialBL}, we get 
		
		\[
		\begin{split}
			\V'_R(t)&\lesssim 8K(\ub)+R^{-4}\sumk\Vert \nabla u_k\Vert_{L^2}^2+R^{-4}\sumk \ak \Vert u_k\Vert_{L^2}^2-\frac{16}{d-4}\RE\int \left(\Delta\phi_R-d\right)F(\ub)dx,
		\end{split}\]
		where we have used the fact that for $r\leq R$, $\phi_R(r)=r^2/2$ and hence $\Delta \phi_R-d\equiv0$. Now, Lemma \ref{lemma22} and Strauss inequality combined yields, for $\kl$, 
		\[
		\int_{|x|\geq R}|u_k|^{\frac{2d}{d-4}}\leq R^{-\frac{4(d-1)}{d-4}}\Vert u_k\Vert_{L^2}^{\frac{2d-4}{d-4}}\Vert u_k\Vert_{L^2}^{\frac{4}{d-4}},
		\]
		which implies
		\[
		\V'_R(t)\leq 8K(\ub)-\frac{16d}{d-4}P(\ub)+ O\left(R^{-4}+R^{-\frac{4(d-1)}{d-4}}\sumk\Vert \nabla u_k\Vert_{L^2}^{\frac{4}{d-4}}+R^{-2}\sumk \Vert \nabla u_k\Vert_{L^2}^2\right).
		\]
		By de definition of energy \eqref{energy} and energy conservation we arrive at \eqref{virialinequality}.
	\end{proof}
	
	Now we are in position to prove Theorem \ref{teoblowup}
		
	\begin{proof}[Proof of Theorem \ref{teoblowup}]
		\textit{Case 1:} $E(\ub_0)<0$. By Lemma \ref{virialinequality} we have that	
		\[
		\V'_R(t)\leq \frac{16d}{d-4}E(\ub_0)-\frac{32}{d-4}K(\ub)+O\left(R^{-4}+R^{-\frac{4(d-1)}{d-4}}\sumk\Vert \nabla u_k\Vert_{L^2}^{\frac{4}{d-4}}+R^{-2}\sumk \Vert \nabla u_k\Vert_{L^2}^2\right).
		\]
		A simply interpolation argument, combined with mass conservation gives us 
		\[
		\Vert \nabla u_k\Vert_{L^2}^{\frac{4}{d-4}}\lesssim \Vert \Delta u_k\Vert_{L^2}^{\frac{2}{d-4}}\quad\hbox{and}\quad\Vert \nabla u_k\Vert_{L^2}^{2}\lesssim \Vert \Delta u_k\Vert_{L^2},\quad\kl.
		\]
		Since $d\geq5$, we have $\frac{2}{d-4}\leq2$ which yields
		\[
		\Vert \Delta u_k\Vert_{L^2}^{\frac{2}{d-4}}\leq 1+ \Vert \Delta u_k\Vert_{L^2}^{2}\quad\hbox{and}\quad\Vert \Delta u_k\Vert_{L^2}\leq 1+\Vert \Delta u_k\Vert_{L^2}^{2},\quad\kl.
		\]
		Therefore, taking $R$ large enough and absorbing the error term we arrive at
		\begin{equation}\label{blow1}
			\V_R'(t)\leq \frac{8d}{d-4}E(\ub_0)-\frac{16}{d-4}K(\ub)<-C<0.
		\end{equation}
		Also, 
		\begin{equation}\label{blow2}
			|\V_R(t)|\lesssim \sumk \Vert \nabla\phi_R\Vert_{L^\infty}M(\ub)\Vert\nabla u_k\Vert_{L^2}\lesssim R K(\ub)^{1/4}.
		\end{equation}
		Assume that $T=\infty$. This allow us to choose $t_1$ large enough such that \eqref{blow1} implies that $\V_R'(t)\leq 0$ for all $t\geq t_1$. Integrating \eqref{blow1} on $[t_1,t]$, the Fundamental theorem of calculus gives us
		\[
		\V_R(t)\leq -\frac{16}{d-4}\int_{t_1}^t K(\ub(s))ds
		\]
		which, by \eqref{blow2}, implies
		\begin{equation}\label{blow3}
			\V_R(t)\leq -C_0\int_{t_1}^t |\V_R(s)|^4ds.
		\end{equation}
		where $C_0=C(\ub_0, R)$. Now, if we denote by $A(t):=\int_{t_1}^t |\V_R(s)|^4ds$ then  we have that $A(t)$ is non-decreasing and by \eqref{blow3},
		\[
		A'(t)=|\V_R(t)|^4\geq C_0^4 A(t)^4.
		\]
		Thus, for $t_2>t_1$, 
		\[
		\int_{t_2}^t \frac{A'(s)}{A(s)^4}ds\geq C_0^4(t-t_2)
		\]
		which implies
		\[
		A(t)^3\geq \frac{A(t_2)^3}{1-3C_0^4A(t_2)^3(t-t_2)}.
		\]
		Hence, $A(t)\rightarrow\infty$ as $t\rightarrow t^*:=t_2+\frac{1}{3C_0^4A(t_2)^3}$. Since $\V_R(t)\leq -C_0A(t)$, it follows that $\V_r(t)\rightarrow-\infty$ as $t\rightarrow t^*$. Therefore, \eqref{blow2} implies that $K(\ub(t))\rightarrow\infty$ which is a contradiction with $T = \infty$. Therefore, $T<\infty$, and by Corollary \ref{standblowup} we have that $S_{[0,T)}(\ub)=\infty.$
		
		\quad\newline
		\textit{Case 2:} $E(\ub)\geq0 $ and \eqref{blowupcondition} holds. Define the function $G(y)=\frac{1}{2}y-C_{\rm opt}y^{\frac{d}{d-4}}$, where $C_{\rm opt}$ is given in \eqref{optcte}. It is easy to see that $G(y)$ has a unique global maximum given by
		\[
		G(y_{\rm max})=\frac{2}{d}y_{\rm max}\quad\hbox{with}\quad y_{\rm max}^{\frac{4}{d-4}}=\left(\frac{d-4}{2d} \right)C_{\rm opt}^{-1}
		\]
		By Lemma \ref{pohozaevident}, we get
		\[
		y_{\rm max}^{\frac{4}{d-4}}=\left(\frac{d-4}{2d} \right)J(\psib)
		\]
		which implies that $y_{\rm max}=K(\psib)$ and $F(y_{\rm max})=E(\psib)$. By \eqref{blowupcondition} and a continuity argument we deduce that
		$K(\ub(t))>K(\psib)$ for all $t\in[0,T)$ (see Hespanha and Scarpelli \cite{HS}, Lemma 2.6) . Choosing $\eta>0$ such that $E(\ub_0)<(1-\eta)E(\psib)$ and using once more Lemma \ref{pohozaevident}, we get
		\[
		E(\ub_0)<(1-\eta)\frac{2}{d}K(\ub(t)),\quad \forall t\in[0,T).
		\]
		Using Lemma \ref{virialinequality} together with the uniform lower bound $K(\ub(t))>K(\psib)$ and choosing $R$ large enough, we arrive at
		\[
		\V'_R(t)\leq -\eta\frac{32}{d-4}K(\ub(t)),\quad \forall t\in[0,T).
		\]
		Arguing as before, we get that $T<\infty$. By Corollary \ref{standblowup} we have that $S_{[0,T)}(\ub)=\infty.$
		
	\end{proof}

\section{Scattering}
This section is devoted to prove the scattering result and, consequently, the global well-posedness in $\Dz$. This will be done by the concentration-compactness and rigidity method. The argument here is very similar to the one in \cite{MXZ}, so we will focus on the proofs that the nonlinearities play important roles, which is the main difference in our work.  Before start the procedure, we need the local well-posedness in $\Dz$. This can be done as in the NLS case. Using the same argument as in \cite[Chapter 3]{killip3}, one can take a initial data in the inhomogeneous Sobolev space $\mathbf{H}^2(\R^d)$ and get a local solution by using Corollary \ref{localexi}. Once we have the solution, Lemma \ref{stability} plays a fundamental role in the argument. In addition to ensuring continuous dependence on the initial data, it also allows us to work with the initial data in the homogeneous Sobolev space $\Dz$ by using an approximating procedure, since $\Dz$ functions is well approximated by $\mathbf{H}^2(\R^d)$ functions. The procedure is standard and for this reason we will omit the details. We summarize the result in the following Theorem.  

\begin{theorem}(Local well-posedness). Let $5\leq d\leq 16$. Given $\ub_0\in\Dz$, there exists a unique maximal-lifespan solution $\ub:I\times\R^d\rightarrow\C$ to \eqref{SISTB} with initial data $\ub(0)=\ub_0$. In addition, the solutions satisfies the following properties:
\begin{itemize}
\item[(i)] (Local existence) $I$ is a open neighborhood of $0$. 
\item[(ii)] (Continuous dependence) If $\ub_0^{(n)}$ is a sequence converging to $\ub_0\in\Dz$ and $\ub^{(n)}:I_n\times\R^d\rightarrow\C$ are the associated solutions to \eqref{SISTB}, then $\ub^{(n)}$ converge locally uniformly to the corresponding solution $\ub$ with initial data $\ub(0)=\ub_0$.
\item[(iii)] (Blow-up criterion) If $\sup(I)$ is finite, the $\ub$ blows-up forward in time. If $\inf(I)$ is finite, then $\ub$ blows-up backward in time.
\item[(iv)] (Small data global existence) If $K(\ub_0)$ is sufficiently small (depending on $d$), then $\ub$ is a global solution which does not blow-up either forward or backward in time. In this case, $S_\R(\ub)\lesssim K(\ub)^{\frac{d+4}{d-4}}$.
\end{itemize}
\end{theorem}
 We start by proving the existence of a critical solutions that is almost periodic modulo symmetries. 
 \subsection{Reduction to almost periodic solutions}
 
 We start this Section noticing that for any $0\leq K_0\leq K(\boldsymbol{\psi})$, the function
 $$
 L(K_0):=\sup\left\{S_I(\mathbf{u});\, \mathbf{u}:I\times\mathbb{R}^d\longrightarrow\mathbb{C}^l \hbox{ is a solution to \eqref{SISTB}  such that }\sup_{t\in I}K(\mathbf{u}(t))\leq K_0\right\}.
 $$
 Therefore, $L:[0,K(\boldsymbol{\psi})]\longrightarrow[0,\infty]$ is a nondecreasing function and since the ground states  are time independent we get  $L(K(\boldsymbol{\psi}))=\infty$, for any $\boldsymbol{\psi}\in\mathcal{G}$. Also,  from Corollary \ref{smalldata}, if $K_0$ is sufficiently small then we deduce that $L(K_0)\lesssim K_0^{\frac{d+4}{d-4}}$. Moreover, Lemma \ref{stability} yields that for any solution of \eqref{SISTB}, it is possible to find another solution with kinetic energy and scattering size close to the initial one; this implies the continuity of $L$. Hence, there must exist a \textit{critical energy}, denoted by $K_c$, such that
 \begin{equation*}\label{ps42}
 	L(K_0)\left\{\begin{array}{lc}
 		<\infty,& \mbox{if}\quad K_0<K_c,\\
 		=\infty,& \mbox{if}\quad K_0\geq K_c.
 	\end{array}\right.
 \end{equation*}
 In particular, if $\mathbf{u}:I\times\mathbb{R}^d\longrightarrow \mathbb{C}^l$ is a maximal solution such that  $\sup_{t\in I}K(\mathbf{u}(t))< K_c$, then, from Proposition \ref{standblowup}, $\mathbf{u}$ is global and
 $$
 S_{\mathbb{R}}(\mathbf{u})\leq L\left(\sup_{t\in\mathbb{R}}K(\mathbf{u}(t))\right)<\infty.
 $$
 The next result is essential to reach the goal of this subsection.
 
 \begin{proposition}[Palais-Smale condition]\label{palais}
 	Let $\mathbf{u}_n:I_n\times\mathbb{R}^d\longrightarrow\mathbb{C}^l$ be a sequence of solutions of \eqref{SISTB} such that
 	\begin{equation}\label{ps31}
 		\limsup_{n\rightarrow\infty}\left(\sup_{t\in I_n}\Vert \Delta\mathbf{u}_n(t)\Vert_{\mathbf{L}_x^2}^2\right)=K_c.
 	\end{equation}
 	Assume there is a sequence of times $(t_n)$ with $t_n\in I_n$ obeying
 	$$
 	\lim_{n\rightarrow\infty}S_{\geq t_n}(\mathbf{u}_n)=\lim_{n\rightarrow\infty} S_{\leq t_n}(\mathbf{u}_n)=\infty.
 	$$
 	Then, the sequence $(\mathbf{u}_n(t_n))$ has a subsequence that converges in  $\dot{\mathbf{H}}^2_x$ modulo symmetries. 
 \end{proposition}
 \begin{proof}
 	We will omit the details since the proof of this result is essentially the same as the one in \cite[Proposition 3.1]{killip} (see also, \cite[Proposition 5.6]{killip3} and \cite[Proposition 7.2, page 263]{visan}). The only differences is that now we need to work with vector functions and in $\Dz$ instead of ${\dot{\mathbf{H}}}^1$. We also need Lemma \ref{gradsum} to work with the nonlinearities part in the proof.  One can also check a fourth order Schrödinger equation version in \cite[Section 6]{MXZ}.
 \end{proof}
 The main result of this subsection is the following.
  \begin{theorem}\label{reduction}
 	Let $5\leq d\leq 16$ be such that Theorem \ref{spacetimebounds} fails. Then there exists a maximal-lifespan solution $\ub:I\times\R^d\rightarrow \C^l$ to \eqref{SISTB} such that 
 	\begin{equation}
 		\sup_{t\in I} K(\ub(t))<K(\psib),
 	\end{equation}
 	$\ub$ is almost periodic modulo symmetries and blows-up in time. In addition, $\ub$ has minimal kinetic energy among all blow-up solutions, that is, 
 	$$
 	\sup_{t\in I}K(\ub(t))\leq \sup_{t\in J} K(\vb),
 	$$
 	for all maximal-lifespan solutions $\vb:I\times\R^d\rightarrow\C^l$ that blows-up at least in one direction.
 \end{theorem}
 \begin{proof}
 	Since the nonlinearities plays no role in this proof we will omit the details. The reader can check the original proof, for the NLS case, in \cite[Section 3.2]{killip} or the BNLS version in \cite[Theorem 6.1]{MXZ}.
 	 \end{proof}

 \subsection{Coercivities estimates}
In this subsection we establish some adapted versions of the coercivity lemmas appearing in \cite{KM}.

\begin{lemma}[Coercivity I]\label{coercivity1} Assume that $\mathbf{u}_0\in\dot{\mathbf{H}}_{x}^{2}(\mathbb{R}^d)$ and let $\mathbf{u}$ be a solution of \eqref{SISTB} with maximal existence interval $I$. Let $\boldsymbol{\psi}\in\mathcal{G}$ be a ground state. Suppose that 
	$	\E(\mathbf{u}_0)<(1-\tilde{\delta})\E(\boldsymbol{\psi}) $ and $K(\mathbf{u}_0)<K(\boldsymbol{\psi}),$
		then there exists $\tilde{\delta}_1=\tilde{\delta}_1(\tilde{\delta}) $ such that
		$$
		K(\mathbf{u}(t))<(1-\tilde{\delta}_1)K(\boldsymbol{\psi}),
		$$
		for all $t\in I$. 
		
\end{lemma}

\begin{proof}
	From the conservation of the energy  and \eqref{sobolevtype}, we deduce
	\begin{equation}\label{q28}
		K(\mathbf{u}(t))\leq E(\mathbf{u}_0)+C_{\rm opt}K(\mathbf{u}(t))^{\frac{d}{d-4}},\quad \forall t\in I.
	\end{equation}
	Let $G(t)=K(\mathbf{u}(t))$, $a=E(\mathbf{u}_0)$, $b=2C_{\rm opt}$ and $q={\frac{d}{d-4}}$ in Lemma \ref{lema comp}. Using \eqref{optimal}, we get
	$$
	\gamma=(bq)^{-\frac{1}{q-1}}=\left(\frac{d-4}{2d}\frac{1}{C_{\rm opt}}\right)^{\frac{d-4}{4}}=K(\psib).
	$$ Lemma \ref{lema comp} implies the result.
\end{proof}

\begin{lemma}[Coercivity II]\label{qCE} 
	Under hypothesis of Lemma \ref{coercivity1} we have
		If $K(\mathbf{u}_0)<K(\boldsymbol{\psi}),$ 	then there exists $\delta'=\delta'(\tilde{\delta})>0$ such that
		\begin{equation}\label{M13}
			K(\mathbf{u}(t))-\frac{2d}{d-4}P(\mathbf{u}(t))\geq \delta'K(\mathbf{u}(t)),
		\end{equation} 
		for all $t\in I$. Moreover,  $E(\ub)\geq 0$.
		
\end{lemma}

\begin{proof}
	 Using \eqref{optimal} Lemma \ref{lema comp} and Lemma \ref{coercivity1}, we deduce
	$$
	1-\frac{2d}{d-4}\frac{P(\mathbf{u}(t))}{K(\mathbf{u}(t))}\geq 1-\frac{2d}{d-4}C_{\rm opt}K(\mathbf{u}(t))^{\frac{4}{d-4}}=1-\left[\frac{K(\mathbf{u}(t))}{K(\boldsymbol{\psi})}\right]^{\frac{4}{d-4}}\geq 1-(1-\tilde{\delta}_1)^{\frac{2}{d-4}}=:\delta'.
	$$
	Multiplying both sides by $K(\mathbf{u}(t))$ we get the result. To prove that the energy in non negative, we observe that if $E(\ub)\geq E(\psib)$ the result is trivial since $E(\psib)>0$. Otherwise, since
	\[
	E(\ub)=K(\ub(t))-2P(\ub(t))=\frac{2}{d}K(\ub(t))+\frac{d-4}{d}\left[K(\ub(t))-\frac{2d}{d-4}P(\ub(t))\right],
	\]
	inequality \eqref{M13} implies that $E(\ub)\geq 0$.
\end{proof}

\begin{lemma}[Energy trapping]\label{qET} Let $\mathbf{u}$ be a solution of \eqref{SISTB} with maximal existence interval $I$ and initial data $\mathbf{u}_0$. If $E(\mathbf{u}_0)\leq (1-\delta)E(\boldsymbol{\psi})$ and $K(\mathbf{u}_0)\leq (1-\delta')K(\boldsymbol{\psi})$, then
	\begin{equation*}\label{g312}
		K(\mathbf{u}(t)) \sim E(\mathbf{u}(t)),\quad\forall t\in I.
	\end{equation*}
\end{lemma}
\begin{proof}
	By \eqref{optimal} and $E(\mathbf{u}_0)\leq (1-\delta)E(\boldsymbol{\psi})$ we obtain
	\begin{equation*}
		\begin{split}
			E(\mathbf{u}(t))&\leq K(\mathbf{u}(t))+2|P(\mathbf{u}(t))|\\
			&\leq K(\mathbf{u}(t))+ 2C_{\rm opt}|K(\mathbf{u}(t))|^{\frac{d}{d-4}}\\
			&\leq \left(1+2C_{\rm opt}[(1-\tilde{\delta}_1)K(\boldsymbol{\psi})^{\frac{4}{d-4}}]\right)K(\mathbf{u}(t)).
		\end{split}
	\end{equation*}
	On the other hand, 
	\begin{equation*}
		\begin{split}
			E(\mathbf{u}(t))&\geq \frac{2}{d}K(\mathbf{u}(t))+\frac{d-2}{d}[K(\mathbf{u}(t))-\frac{2d}{d-2}P(\mathbf{u}(t))]\\
			&\geq \frac{2}{d}K(\mathbf{u}(t))+\frac{d-2}{d}\delta' K(\mathbf{u}(t))\\
			&=\frac{2}{d}\left(1+\frac{d-2}{d}\delta'\right)K(\mathbf{u}(t)).
		\end{split}
	\end{equation*}
	Combining both inequalities, we get the result.
\end{proof}

\subsection{The enemies}

The next step to achieve our goal is to prove that the solution $\ub_c$ obtained in Theorem \ref{reduction} cannot exist. In order to do this, we will see that this solution satisfy some extra properties with respect to the scale functions $N(t)$. The following result was first obtained in \cite{killip2} and \cite{killip} for the Schr\"odinger equation and in \cite{MXZ} proved a similar result for the fourth order Shcrödinger equation. The proof does not rely on the specific form of the nonlinearity, so will be omitted. The reader can check the details on \cite[Theorem 1.17]{killip} ou \cite[Theorem 7.1]{MXZ}. The result is the following
\begin{proposition}[The enemies]\label{qP114} Suppose that Theorem \ref{spacetimebounds} fails. Then there exists a maximal solution $\mathbf{u}_c:I_c\times\mathbb{R}^d\longrightarrow\mathbb{C}^l$, which is almost periodic modulo symmetries and satisfies
	\begin{equation}\label{q114}
		S_{I_c}(\mathbf{u}_c)=\infty\quad \hbox{and} \quad \sup_{t\in I_c}K(\mathbf{u}_c(t))<K(\boldsymbol{\psi}).
	\end{equation}
	Moreover, the time interval $I_c$ and the scale function $N(t)$ satisfy one of the three following scenarios:
	\begin{itemize}
		\item[(i)] We have $|\inf I_c|<\infty$ or $|\sup I_c|<\infty$;\\
		\item[(ii)] We have $I_c=\mathbb{R}$ and
		$$
		N(t)=1,\quad \forall t\in \mathbb{R};
		$$
		\item[(iii)] We have $I_c=\mathbb{R}$ and
		$$
		\inf_{t\in\R}N(t)\geq 1,\quad \hbox{and}\quad \limsup_{t\rightarrow\infty}N(t)=\infty.
		$$
	\end{itemize}
	
\end{proposition}

In the literature,(see, for instance, \cite{killip}) the three scenarios are known, respectively, as \textit{finite-time blow-up}, \textit{soliton-like solution} and \textit{low-to-high frequency cascade}. The whole idea is to prove that none of this kind of solutions can exist and then, arrive at a contradiction with the existence of $\ub_c$. The first result in this direction is to preclude the existence of finite-time blow-up solution, but before we state the result, we need the following Lemma.

\begin{lemma}\label{MXZ 410}
	Let $\ub:I\times\R^d\rightarrow\C^l$ be a maximal-lifespan solution to \eqref{SISTB} that is almost periodic modulo symmetries with frequency scale function $N:I\rightarrow\R^+$. If $\sup I<\infty$, then
	\begin{equation}\label{bu413}
		\liminf_{t\nearrow \sup I}N(t)=\infty.
	\end{equation}
	A similar results holds if $|\inf I|<\infty$.
\end{lemma}
\begin{proof}
Since is related to almost periodic modulo symmetries condition and does not rely on the nonlinearities, we will omit it. The details can be find as part of the proof of Theorem 5.1 in \cite{killip} or in \cite[Corollary 3.7]{killip2}.

\end{proof}

Now we are able to preclude the finite-time blow-up solution. The idea of the proof is analogous as the one in \cite[Theorem 5.1]{killip} and since the nonlinearities does not play a important role, we will just sketch it for completion.  The result read as follows.
\begin{theorem} Let $5\leq d\leq 16$. There are no maximal radial almost periodic modulo symmetries solutions obeying which blow-up in finite time, in the sense of Proposition \ref{qP114}.
\end{theorem}
\begin{proof}[Sketch of the proof]
Suppose, contrary to our claim, that there exist a maximal radial solution $\ub_c:I_c\times\R^d\rightarrow\C^l$ that is almost periodic modulo symmetries and blows-up in finite time. We assume that $\sup I<\infty$, the other case is proved in the same way. As a consequence of \eqref{bu413} and Hölder's inequality we get, for $\mathbf{u}_c=(u_{c1},\ldots,u_{cl})$, that
\begin{equation}\label{bu415}
	\limsup_{t\nearrow\sup I_c}\int_{|x|\leq R}|u_{ck}|^2dx=0,\quad \forall R>0,\, k=1,\ldots,l.
\end{equation}
Now, for $t\in I$, we define
$$
V_R(t):=\int \phi\left(\frac{|x|}R\right)\left(\sumk \frac{\alpha_k^2}{\gamma_{k}}|u_{ck}|^2\right)dx,
$$
where $\phi$ is a smooth radial function obeying $\phi(r)=1$ for $r\leq 1$ and $\phi(r)=0$ for $r\geq 2$. Thus, \eqref{bu415} yields 
\begin{equation}\label{MXZ66}
	\limsup_{t\nearrow\sup I_c} V_R(t)=0,\quad \forall R>0.
\end{equation}
On the other hand, 
\[
V_R'(t)=\sumk \alpha_k\left( -2\IM\int \alpha_k\Delta\phi\left(\frac{|x|}{R}\right)\bar{u}_{ck}\Delta \uck dx-2\IM\int\nabla\phi\left(\frac{|x|}{R}\right)\nabla\bar{u}_{ck}\Delta \uck dx\right).
\]
Hence, by Hardy's inequality (see \cite[Lemma A.2]{tao3}) and \eqref{q114}, we get
\begin{equation*}
	\begin{split}
		|V'_R(t)|\lesssim K(\ub_c)\lesssim K(\psib).
	\end{split}
\end{equation*}
where the implicit constant is independent of $R>0$. Applying the fundamental theorem of calculus, using \eqref{MXZ66} and letting $t_2\nearrow\sup I_c$
leads us to
\[
V_R(t_1)\lesssim |\sup I_c-t_1|K(\psib).
\]
If we take $R\rightarrow\infty$, by the conservation laws we obtain $\ubc\in \Lb^2(\R^d)$. Finally, taking $t_1\nearrow\sup I_c$, we conclude that $\ub_c=0$, contradicting \eqref{q114}.
\end{proof}
In order to exclude the solitons and cascade we will need some tools. We first show a virial-type inequality as follows.
\begin{lemma}\label{lemma83}
	Let $\phi$ be a smooth function with $\phi(r)=1$ if $r\leq 1$ and $\phi(r)=0$ if $r\geq2$. Set 
	\[
	z_R(t)=\sumk \IM\int \alpha_k x\phi\left(\frac{|x|}{R}\right)\cdot \nabla\bar{u}_ku_kdx.
	\]
	Then, 
	\begin{equation}\label{70}	
	z_R'(t)\geq 4\left(K(\ub)-\frac{2d}{d-4}P(\ub)\right)-O\left(\sumk\int_{R\leq |x|\leq2R}\left(\frac{|u_k||\Delta u_k|}{R^2}+\frac{|\nabla u_k||\Delta u_k|}{R}+|\Delta u_k|^2\right)dx\right).
\end{equation}
\end{lemma}
\begin{proof}
	We start noticing that by using integration by parts and $\IM(\bar{z})=-\IM(z)$, we have
\begin{equation}\label{IP}
		 \begin{split}
		 	\sumk\IM \int \alpha_k& x \phis \cdot\nabla(\partial_t\bar{u}_k)u_kdx=\sumk\alpha_k\IM\int\sumjd x_j\phis\partialj(\partial_t\bar{u}_k)u_kdx\\
		 	&=-\sumk\alpha_k \IM\int\sumjd \partialj\left(x_j\phis u_k\right)\partial_t\bar{u}_kdx\\
		 	&=-\sumk\alpha_k \IM\left[\int\sumjd\partialj\left(x_j\phis\right){u}_k\partial_t\bar{u}_kdx+\int\sumjd x_j\phis\partialj{u}_k\partial_t\bar{u}_kdx\right]\\
		 		&=\sumk\alpha_k \IM\int \nabla\cdot \left(x\phis\right)\bar{u}_k\partial_tu_kdx+\sumk\alpha_k\IM\int x\phis\cdot \nabla\bar{u}_k\partial_tu_kdx.
		 \end{split}
		 \end{equation}
		 Therefore, differentiating $z_R$ with respect to t and using \eqref{IP} we get
		 \[
		 \begin{split}
		 z_R'(t)&= \sumk\IM\int\alpha_k x\phis\cdot\partial_t\left(\nabla\bar{u}_ku_k\right)dx\\
		 	&=\sumk\IM\int\alpha_k x\phis\cdot\nabla(\partial_t\bar{u}_k)u_kdx+\sumk\IM\int\alpha_k x\phis\cdot\nabla\bu_k\partial_tu_k dx\\
		 		&=2\sumk\alpha_k\IM\int x\phis\cdot\nabla\bu_k\partial_tu_kdx+\sumk\alpha_k\IM\int\nabla\cdot\left(x\phis\right)\bu_k\partial_tu_kdx\\
		 		&:=A+B\\
		 		 \end{split}
		 \]

	Let us compute $A$.	Using \eqref{SISTB}, integration by parts and Lemma \ref{lemma22} (iii)
	
	\begin{equation}\label{A}
		\begin{split}
			A&=2\sumk\gk \RE\int x\phis\cdot\nabla\bu_k \Delta^2u_k dx-2\sumk\RE\int x\phis\cdot\nabla \bu_k f_k(\ub)dx\\
			&=2\sumk\gk \RE\int \Delta\left(x\phis\cdot\nabla\bu_k\right)\Delta u_k-2 \RE\int x\phis\cdot \nabla F(\ub)dx\\
			&:=A_1-A_2\\
					\end{split}
	\end{equation}
		 
For $A_1$, first notice that since $\phis$ is radial, a straightforward calculation gives us
\begin{equation}\label{phigradlap}
	\nabla\phis=\phi'\xR\frac{x}{R|x|}\quad\hbox{and}\quad\Delta\phis=\frac{1}{R^2}\phi'\xR+\frac{(d+1)}{R|x|}\phi'\xR.
\end{equation} 		 
Also, using integration by parts, for $\kl$, 
\begin{equation}\label{gradlap}
	2\RE\int x\phis\cdot\nabla(\Delta \bu_k)\Delta u_kdx=-\int\phi'\xR\frac{|x||\Delta u_k|^2}{R}dx- d\RE\int\phis|\Delta u_k|^2dx.
	\end{equation}
Therefore, \eqref{phigradlap} and \eqref{gradlap} yields

\begin{equation}\label{A1}
	\begin{split}
		A_1	&=2\sumk\gk \RE\int\left[2\phis\Delta \bu_k+2\nabla\bu_k\cdot\nabla\phis+x\phis\cdot\nabla(\Delta\bu_k)\right.\\
		&\quad+\left. 2x\cdot\nabla\phis\Delta\bu_k+x\cdot\nabla\bu_k\Delta\phis    \right]\Delta u_kdx\\
		&= (4-d)\sumk\gk\int \phis|\Delta u_k|^2dx+2\sumk\gk\RE\int \phi'\xR \frac{(d+2)x\cdot\nabla\bu_k\Delta u_k}{R|x|}dx\\
		&\quad+3\sumk\gk\int\phi'\xR\frac{|x||\Delta u_k|^2}{R}dx+2\sumk\gk\RE\int\phi''\xR\frac{x\cdot\nabla \bu_k\Delta u_k}{R^2}dx.
	\end{split}
\end{equation}	 
		For $A_2$, using once again integration by parts, 
		\begin{equation}\label{A2}
		\begin{split}
			A_2&= -2\RE\int \nabla\left(x\phis\right)\nabla F(\ub)dx=-2d\RE\int \phis F(\ub)dx-2\RE\int \phi'\xR \frac{|x| F(\ub)}{R}dx.\\
		\end{split}
		\end{equation}

		 Before we calculate B, we notice that
		 $$
		 \nabla\cdot\left(x\phis\right)=\sumjd \partialj \left(x_j\phis\right)=d\phis+\frac{|x|}{R}\phi'\xR.
		 $$
		 Therefore, using \eqref{SISTB}
		 \[
		 \begin{split}
		 	B&=\sumk\alpha_k\IM\int\nabla\cdot\left(x\phis\right)\bu_k\partial_tu_kdx\\
		 	&=\sumk\alpha_k\RE\int\nabla\cdot\left(x\phis\right)\bu_k\left(\gamma_k\Delta^2u_k-f_k(\ub)\right)\\
		 	&=d\sumk \gamma_k \RE \int\phis\bu_k\Delta^2u_kdx+\sumk\gamma_k\RE\int\frac{|x|}{R}\phi'\xR\bu_k\Delta^2u_kdx\\
		 	&\quad-d\sumk \RE \int\phis\bu_kf_k(\ub)dx-\sumk\RE \int \frac{|x|}{R}\phi'\xR\bu_kf_k(\ub)dx\\
		 	&:=B_1+B_2-B_3-B_4.
		 \end{split}
		 \]
		Using integration by parts twice we have
		\[
		\begin{split}
			B_1&=d\sumk\gamma_k\RE\int \phis\bu_k\Delta^2u_kdx\\
			&=d\sumk\gamma_k\RE\int\sumjd \frac{\partial^2}{\partial x_j^2}\left(\phis \bu_k\right)\Delta u_kdx\\
			&=d\sumk\gamma_k\RE\int \frac{(d-1)}{R|x|}\phi'\xR\bu_k\Delta u_kdx+ d\sumk\gamma_k\RE\int\phi''\xR\frac{\bu_k\Delta u_k}{R^2}dx\\
			&\quad +2d\sumk\gamma_k\RE\int\phi'\xR\frac{x\cdot\nabla\bu_k\Delta u_k}{R|x|}dx+d\sumk\gamma_k\int\phi|\Delta u_k|^2dx.\\
		\end{split}
		\]
		In the same way,
		\[
		\begin{split}
			B_2&=\sumk\gamma_k\RE\int \sumjd \frac{\partial^2}{\partial x_j^2}\left(\phi'\xR\frac{|x|\bu_k}{R}\right)\Delta u_kdx\\
			&=\sumk\gamma_k\RE\int\phi'\xR \frac{(d-1)\bu_k\Delta u_k}{R|x|}dx+\sumk\gamma_k\RE\int\phi''\xR\frac{(d+1)\bu_k\Delta u_k}{R^2}dx\\
							&\quad +2\sumk\gamma_k\RE\int\phi'\xR\frac{x\cdot\nabla\bu_k\Delta u_k}{R|x|}dx+\sumk\gamma_k\RE\int\int \phi'''\frac{\bu_k\Delta u_k}{R^3|x|}dx\\
							&\quad 2\sumk\gamma_k\RE\int\phi''\xR\frac{x\cdot\nabla\bu_k\Delta u_k}{R^2}dx+\sumk\gamma_k\int\phi'\xR\frac{|x||\Delta u_k|^2}{R}dx.\\
						\end{split}
		\]
Gathering all together, we have

\[
\begin{split}
	B&=d\sumk\int\phis \left(\gamma_k|\Delta u_k|^2-\bu_k f_k(\ub)\right)dx+2d\sumk\gamma_k\RE\int \phi'\xR \frac{x\cdot\nabla\bu_k\Delta u_k}{R|x|}dx\\
	&\quad+d\sumk\gamma_k\RE\int\left[\phi''\xR\frac{\bu_k}{R^2}+\phi'\xR\frac{(d-1)x\cdot\nabla\bu_k}{R|x|}\right]\Delta u_kdx\\
	&\quad+\sumk\int\phi'\xR\frac{|x|(\gamma_k|\Delta u_k|^2-\bu_kf_k(\ub))}{R}dx\\
	&\quad +\sumk\gamma_k\RE\int\left[\phi'\xR\frac{(d-1)}{R|x|}+\phi''\xR\frac{(d+1)}{R^2}+\phi'''\xR\frac{|x|}{R^3}\right]\bu_k\Delta u_kdx\\
	&\quad 2\sumk\gamma_k\RE\int\left[\phi'\xR \frac{1}{R|x|}+\phi''\xR\frac{1}{R^2}\right]x\cdot\nabla\bu_k\Delta u_kdx.
\end{split}
\]
	Now, \eqref{70} follows	by using Lemma \ref{lemma22} and noticing that $\phi',\,\phi''$, and $\phi'''$ are supported in $\{x:\,R\leq |x|\leq 2R\}$.
\end{proof}

Now we are in position to preclude the existence of solitons and cascade.

\begin{theorem}
	There are no global radial solution to \eqref{SISTB} that are solitons or low-to-high cascades in the sense of Proposition \ref{qP114}.
\end{theorem}
\begin{proof} 
Suppose that $\ub$ is a  global almost periodic modulo symmetries solution with frequency scale function $N(t)\geq 1$ for all $t\in\R$. Thus, Remark \ref{almost periodic} together with Hardy's inequality yields that for any $\epsilon>0$, there exists $R(\epsilon)>0$ such that, for all $t\in[0,\infty)$
\begin{equation}\label{M67}
	\sumk \int_{|x|>R(\epsilon)}\left(|\Delta u_k|^2+\frac{|\nabla u_k|^2}{|x|^2}+\frac{|u_k|^2}{|x|^4} \right)dx\leq \epsilon.
\end{equation}
  
 On the other hand, by Lemma \ref{qET} and \eqref{M13} gives us
\begin{equation}\label{68}
4\left(K(\ub)-\frac{2d}{d-4}P(\ub)\right)\geq \delta' K(\ub_0).
\end{equation}
Taking $\epsilon=\epsilon_0K(\ub_0)$ and combining \eqref{M67} with \eqref{68}, we get that there exists $R_0>0$ such that, for all $t\in[0,\infty]$, we have
\begin{equation}\label{M68}
	4\left( \sumk\int_{|x|\leq R_0} |\Delta u_k|^2dx-\frac{2d}{d-4}\RE\int_{|x|\leq R_0} F(\ub)dx \right)\gtrsim K(\ub_0).
\end{equation}

Thus, for $R$ large enough,  Lemma \ref{lemma83} yields
\[
z_R'(t)\gtrsim K(\ub_0).
\]
Integrating in $t$ gives us
\[
z_r(t)-z_R(0)\gtrsim t K(\ub_0).
\]
By the definition of $z_R(t)$, 
\[
|z_R(t)-z_R(0)|\lesssim 2R^4K(\psib),
\]
which is a contradiction for $t$ large.

\end{proof}

\subsection*{Acknowledgment}
M. H. is supported by Conselho Nacional de Desenvolvimento Científico e Tecnológico - CNPq. R.S. is supported by Fundação de Amparo a Pesquisa do Estado de Minas Gerais - FAPEMIG. The authors thank Professor Luiz Gustavo Farah for his support and valuable suggestions during the preparation of this work.


\begin{thebibliography}{99}
\bibitem{AA}
\newblock Akhmediev, N., Ankiewicz, A.:
\newblock \emph{Novel soliton states and bifurcation phenomena in nonlinear fiber couplers},
\newblock Phys. Rev. Lett. 70, 2395-2398, (1993).

\bibitem{AP} 
\newblock Ambrosetti, A., Prodi, G.:
\newblock \emph{A primer of nonlinear analysis}, 
\newblock volume 34 of Cambridge Studies in Advanced Mathematics. Cambridge University Press, Cambridge, 1995.

\bibitem{BS}
\newblock Badiale, M., Serra, E.:
\newblock \emph{Semilinear elliptic equations for beginners: Existence results
	via the variational approach},
\newblock Universitext. Springer, London, 2011.

\bibitem{AKS}
\newblock Ben-Artzi, M., Koch, H., Saut, J. C.:
\newblock \emph{Disperion estimates for fourth order Schrödinger
	equations},
	\newblock C.R.A.S., 330, Série 1, 87-92, (2000).
	
\bibitem{BL17}
\newblock Boulenger, T., Lenzmann, E.:
\newblock \emph{Blowup for biharmonic NLS},
\newblock Ann. Sci. Éc. Norm. Supér. \textbf{50}(4), 503–544 (2017).


	\bibitem{bourgain}
	\newblock Bourgain, J.:
	\newblock \emph{Global well-posedness of defocusing 3D critical NLS in the radial case},
	\newblock JAMS, 12, 145-171, (1999).
	
	
	
	\bibitem{BL} 
	\newblock Brezis, H.,  Lieb, E.:
	\newblock \emph{A relation between pointwise convergence of function and convergence of functional},
	\newblock Proc. Amer. Math. Soc., v. 88, no. 3, 486--490,  (1983).
	
	\bibitem{bauer}
	\newblock Bauer, H.:
	\newblock \emph{Measure and integration theory: Translated from the German by Robert B. Burckel},
	\newblock De Gruyter Studies in Mathematics, vol. 26. Walter de Gruyter \& Co, Berlin (2001).
	
	
	
	
	\bibitem{cazenave}  Cazenave, T.: \emph{Semilinear Schrödinger Equations}, Amer. Math. Soc.,  Courant Lecture Notes in Mathematics, v. 10, Providence, RI, (2003).
	
	\bibitem{cazetal}
	\newblock Cazenave, T., Fang, D., Han, Z.:
	\newblock \emph{Continuous dependence for NLS in fractional order spaces},
	\newblock Ann. Inst. H. Poincaré Anal. Non Linéaire, v. 28(1), p. 135-147, (2011).
	
	\bibitem{demengel}
	\newblock Demengel, F., Demengel, G.: 
	\newblock \emph{Functional spaces for the theory of elliptic partial differential equations},
	\newblock Universitext. Springer, London; EDP Sciences, Les Ulis, (2012).

\bibitem{EG}
\newblock Evans, L. C., Gariepy, R. F.:
\newblock \emph{Measure theory and fne properties of functions,}
\newblock CRC Press, Boca Raton, FL, 1992. (Studies in Advanced Mathematics).


\bibitem{FH}
\newblock Farah, L. G., Hespanha, M.:
\newblock \emph{The focusing energy-critical nonlinear Schrödinger system with power-type growth nonlinearities in the radial case}

\bibitem{FIP}
\newblock Fibich, G., Ilan, B. and Papanicolaou, G.:
\newblock \emph{Self-focusing with fourth order dispersion},
\newblock SIAM J. Appl. Math. 62, No 4, 1437-1462, (2002).

\bibitem{FM}  Flucher, M.,  Muller, S.: \emph{Concentration of low energy extremals}, Ann. Inst. H. Poincaré Anal. Non Linéare, v. 16, no. 3, p. 269--298,  (1999).

\bibitem{Folland}
\newblock Folland, G. B.:
\newblock \emph{Real analysis: Modern techniques and their applications},
\newblock John Wiley $\&$ Sons, Inc., New York, 1999. (Pure and Applied Mathmatics).


	\bibitem{GMO} Gerard, P., Meyer, Y., Oru, F.: \emph{Inégalités de Sobolev précisées}, Séminaire sur les Équations aux Dérivées Partielles, 1996-1997, Exp. No. IV, École Polytech. (1997),   11 pp.

\bibitem{GW}
\newblock Guo, G., Wang, B.:
\newblock \emph{The global Cauchy problem and scattering of solutions for nonlinear
	Schr¨odinger equations in Hs, }
\newblock Diff. Int. Equ., 15 (2002), 1073–1083.

\bibitem{HHW1}
\newblock Hao, C., Hsiao, L., Wang, B.:
\newblock \emph{Well-posedness for the fourth-order Schrödinger equations},
\newblock J. of Math. Anal. and Appl. 320 (2006), 246–265.

\bibitem{HHW2}
\newblock Hao, C., Hsiao, L., Wang, B.:
\newblock \emph{Well-posedness of the Cauchy problem for the fourth-order
	Schrödinger equations in multi-dimensional spaces},
\newblock J. of Math. Anal. and Appl., 328 (2007),
58–83.

\bibitem{HP}
\newblock Hespanha, M., Pastor, A.:
\newblock  \emph{Blow-Up of Radially Symmetric Solutions for a Cubic NLS-Type System in Dimension 4},
\newblock Studies Applied Mathematics, v. 154, p. e70044, (2025)

\bibitem{HS}
\newblock Hespanha, M., Scarpelli, R.:
\newblock \emph{Blow-up results for Inhomogeneous fourth-order nonlinear Schr\"odinger Equation},
\newblock https://arxiv.org/abs/2507.05518

\bibitem{Karpman}
\newblock Karpman V.I.:
\newblock \emph{Stabilization of soliton instabilities by higher-order dispersion: fourth order
	nonlinear Schrödinger-type equations},
	\newblock Phys. Rev. E 53,  2, 1336-1339, (1996).
	
	\bibitem{Karpmn2}
\newblock Karpman V. I., Shagalov, A. G.:
\newblock \emph{Stability of soliton described by nonlinear Schrödinger type equations with higher-order dispersion},
\newblock Phys D. 144, 194-210, (2000).

\bibitem{KM}  Kenig, D. E., Merle, F.: \emph{Global well-posedness, scattering and blow-up for the energy-critical, 	focusing, non-linear Schr\"odinger equation in the radial case}, Invent. Math., v. 166, no. 3,  p. 645--675, (2006).

\bibitem{keraani} Keraani, S.: \emph{On the defect of compactness for the Strichartz estimates for the Schrödinger equations}, J. Differential Equations, \textbf{175} (2001),  353-392.

	\bibitem{killip2} Killip, R., Tao, T., Visan, M.: \emph{The cubic nonlinear {S}chrödinger equation in two dimensions with radial data}, J. European Math. Soc., \textbf{11} (2008),  1203-1258.

\bibitem{killip} Killip, R., Visan, M.: \emph{The focusing energy-critical nonlinear Schrödinger equation in dimensions five and higher}, American Journal of Mathematics, \textbf{132} (2010), no 2,  361--424.


\bibitem{killip3}  Killip, R., Visan, M.: \emph{Nonlinear Schrödinger Equations at Critical Regularity}, Clay Math. Proc., \textbf{17} (2013),  325-437.

\bibitem{lieb}	Lieb, E., Loss, M.: Analysis, Graduate Studies in Mathematics, v. 14, 2nd edition. American Mathematical Society, Providence (2001).

\bibitem{lions1}  Lions, P. L.:\emph{The concentration-compactness principle in the calculus of variations. The locally compact case}, I, Ann. Inst. H. Poincaré Anal. Non Linéaire, v. 1, no 2, p. 109--145, (1984).

\bibitem{lions2} Lions, P. L.: \emph{The concentration-compactness principle in the calculus of variations. The limit case, part 1.} Rev. Mat. Iberoamericana, v. 1, no. 1, p. 145--201,  (1985).

\bibitem{MXZ}
\newblock Miao, C. Xu, G., Zhao, L.:
\newblock \emph{Global well-posedness and scattering for the focusing
	energy-critical nonlinear Schrödinger equations of fourth order in the radial case},
	\newblock J. Differential Equations 246 (2009), 3715—3749

\bibitem{NoPa} 
\newblock  Noguera N.,  Pastor, A.:
\newblock \emph{On the dynamics of a quadratic {S}chr\"{o}dinger system in dimension {$n =5$}},
\newblock Dyn. Partial Differ. Equ., 1 \textbf{17} (2020).


\bibitem{NoPa2} Noguera, N., Pastor, A.: \emph{A system of Schrödinger equations with general
	quadratic-type nonlinearities}, Commun. Contemp. Math., \textbf{23} (2021), no 4, 2050023, 66 pp.

\bibitem{NoPa3}  
\newblock Noguera, N., Pastor, A.:  
\newblock \emph{Blow-up solutions for a system of Schrödinger equations with general quadratic-type nonlinearities in dimensions five and six}, 
\newblock Calc. Var. Partial Differential Equations, \textbf{61} (2022), no 3, Paper No. 111, 35 pp.		


\bibitem{NoPa4} 
\newblock Noguera, N.,   Pastor, A.:
\newblock  \emph{Scattering of radial solutions for quadratic-type Schr\"{o}dinger systems in dimension five},
\newblock  Discrete Contin. Dyn. Syst. A, \textbf{41} (2021), 3817-3836.

\bibitem{NoPa5} 
\newblock Noguera, N.,   Pastor, A.:
\newblock  \emph{Scattering for quadratic-type {S}chr\"{o}dinger systems in dimension five without mass-resonance},
\newblock  Partial Differ. Equ. Appl., \textbf{2} (2021), no 4, paper No. 60, 30 pp.

\bibitem{pastor2}  Pastor, A.: \emph{Weak concentration and wave operator for a 3D coupled nonlinear Schrödinger system}, J. Math. Phys., v. 56, 021507--1 to 021507--18, (2015).

\bibitem{pastor3}  Noguera, N.,  Pastor, A.: \emph{Blow-up solutions for a system of Schrödinger equations with general quadratic-type nonlinearities in dimensions five and six}, Calc. Var. Partial Differential Equations, v. 61, paper 111, (2022).	

\bibitem{Pausader}
\newblock Pausader, B.: 
\newblock \emph{Global well-posedness for energy critical fourth-order Schrödinger equations in the radial case},
\newblock Dyn. Partial Differ, Equ. 4, 197-255, (2007).

	\bibitem{Pausader2}
	\newblock Pausader, B.: 
	\newblock \emph{The focusing energy-critical fourth-order Schrödinger equation with
		radial data},
		\newblock Discrete Contin. Dyn. Syst. 24, 1275-1292, (2009)
		
		\bibitem{Pausader3}
		\newblock Pausader, B.: 
		\newblock \emph{The cubic fourth-order Schrödinger equation},
		\newblock J. Funct. Anal. 256, 2473-2517, (2010).
	
	
	
	\bibitem{tao} 
	\newblock Tao, T., Visan, M.: 
	\newblock \emph{Stability of energy-critical nonlinear {S}chrödinger equations in high dimensions}, 
	\newblock Electron. J. Differential Equations, \textbf{118} (2005),  1-28.
	
	\bibitem{tao2}
	\newblock Tao, T.:
	\newblock \emph{Global well-posedness and scattering for the higher-dimensional energy-critical nonlinear Schr¨ odinger equation for radial data},
	\newblock New York Journal of Math. 11, 57-80, (2005).
	
	\bibitem{tao3} 
	\newblock Tao, T.:
	\newblock \emph{Nonlinear dispersive equations}, 
	\newblock CBMS Regional Conference Series in Mathematics, Published for the Conference Board of the Mathematical Sciences, Washington, DC; by the American Mathematical Society, Providence, RI, \textbf{106} (2006).
	
	\bibitem{tao4}
	\newblock Tao, T.:
	\newblock \emph{Multilinear Weighted convolutions $L^2$ function, and applications to nonlinear
		dispersive equations}, 
	\newblock Amer. J. Math. 123(5), 839-908, (2001). 
	
	\bibitem{visan} 
	\newblock Koch, H., Tataru, D., Visan, M.: 
	\newblock \emph{Dispersive equations and nonlinear waves. Generalized Korteweg-de Vries, nonlinear Schr\"odinger, wave and Schr\"odinger maps.},
	\newblock  Oberwolfach Seminars, 45. Birkhäuser/Springer, Basel, (2014).
	
\end{thebibliography}
\end{document}